\documentclass[11pt, twoside]{article}

\usepackage{amsmath}
\usepackage{amssymb}
\usepackage{titletoc}
\usepackage{mathrsfs}
\usepackage{amsthm}

\usepackage{indentfirst}
\usepackage{color}
\usepackage{txfonts}

\allowdisplaybreaks

\def\bint{{\ifinner\rlap{\bf\kern.30em--}
\int\else\rlap{\bf\kern.35em--}\int\fi}\ignorespaces}

\def\sbint{{\ifinner\rlap{\bf\kern.32em--}
\hspace{0.078cm}\int\else\rlap{\bf\kern.45em--}\int\fi}\ignorespaces}

\def\vpz{\vphantom}

\def\rr{\mathbb{R}}
\def\rn{\mathbb{R}^n}

\def\cc{\mathbb{C}}
\def\nn{\mathbb{N}}
\def\zz{\mathbb{Z}}

\def\dd{\mathbb{D}}

\def\dz{\delta}

\def\bz{\beta}
\def\gz{{\gamma}}

\def\ls{\lesssim}

\def\fz{\infty}
\def\az{\alpha}

\def\ca{{\mathcal A}}

\def\cd{{\mathcal D}}
\def\cf{{\mathcal F}}
\def\cg{{\mathcal G}}

\def\cm{{\mathcal M}}

\def\cx{{\mathcal X}}
\def\cy{{\mathcal Y}}

\def\lp{{L^p(\mathcal{X})}}

\def\icgg{\mathcal{G}_0^\eta(\beta,\gamma)}
\def\cgg{\mathring{\mathcal{G}}_0^\eta(\beta,\gamma)}

\def\hb{\dot{B}^s_{p,q}(\mathcal{X})}
\def\hf{\dot{F}^s_{p,q}(\mathcal{X})}
\def\ihb{B^s_{p,q}(\mathcal{X})}
\def\ihf{F^s_{p,q}(\mathcal{X})}
\def\hfi{\dot{F}^s_{\infty,q}(\mathcal{X})}

\def\r{\right}
\def\lf{\left}

\def\noz{{\nonumber}}

\def\r{\right}
\def\lf{\left}

\def\oint{\fint}

\def\diam{{\mathop\mathrm{\,diam\,}}}

\def\loc{{\mathop\mathrm{\,loc\,}}}

\DeclareMathOperator*{\esssup}{ess\ sup}
\DeclareMathOperator*{\essinf}{ess\ inf}

\def\eqref#1{(\ref{#1})}

\def\func#1{\mathop{\mathrm{#1}}}
\def\diam{\func{diam}}

\def\ya{y_\alpha^{k,m}}
\def\qa{Q_\alpha^{k,m}}
\def\qo{Q_\alpha^{0,m}}

\def\red{\color{red}}

\newtheorem{theorem}{Theorem}[section]
\newtheorem{lemma}[theorem]{Lemma}
\newtheorem{corollary}[theorem]{Corollary}
\newtheorem{proposition}[theorem]{Proposition}

\theoremstyle{definition}
\newtheorem{remark}[theorem]{Remark}
\newtheorem{definition}[theorem]{Definition}

\numberwithin{equation}{section}

\begin{document}

\title{\bf\Large Pointwise Characterization of Besov and Triebel--Lizorkin Spaces on Spaces of Homogeneous Type
\footnotetext{\hspace{-0.35cm} 2020 {\it Mathematics Subject Classification}. Primary 46E36; Secondary 46E35, 42B25, 30L99.\endgraf
{\it Key words and phrases.} space of homogeneous type, Besov space, Triebel--Lizorkin space,
Haj\l asz--Sobolev space, Haj\l asz--Triebel--Lizorkin space,  grand Triebel--Lizorkin space,
pointwise characterization.\endgraf
This project is partially supported by the National Key Research and Development Program of China
(Grant No.\ 2020YFA0712900) and the National
Natural Science Foundation of China (Grant Nos.\
11971058, 12071197 and 11871100).}}
\date{}
\author{Ryan Alvarado, Fan Wang, Dachun Yang\footnote{Corresponding author,
E-mail: \texttt{dcyang@bnu.edu.cn}/{\red June 21, 2021}/Final version.}\ \  and Wen Yuan}
\maketitle

\vspace{-0.8cm}

\begin{center}
\begin{minipage}{13cm}
{\small {\bf Abstract}\quad In this article, the authors establish the pointwise
characterization of Besov and Triebel--Lizorkin
spaces on spaces of homogeneous type via clarifying the
relationship among Haj\l asz--Sobolev spaces, Haj\l asz--Besov
and Haj\l asz--Triebel--Lizorkin spaces,
grand Besov and Triebel--Lizorkin
spaces, and Besov and Triebel--Lizorkin
spaces. A major novelty of this article is that all results presented
in this article get rid of both the dependence on the
reverse doubling condition of the measure and the metric condition of
the quasi-metric under consideration. Moreover, the pointwise
characterization of the inhomogeneous version is
new even when the underlying space is an RD-space.
}
\end{minipage}
\end{center}

\vspace{0.2cm}


\vspace{0.2cm}

\section{Introduction}\label{intro}

It is well known that Besov and Triebel--Lizorkin
spaces provide a unified frame for the study of many
function spaces and indeed cover many well-known
classical concrete function spaces such as
Lebesgue spaces, Sobolev spaces,
potential spaces, (local) Hardy spaces, and the space of
functions with bounded mean oscillation.
We refer the reader to monographs \cite{bin,t83,t92,t06,w88} for
a comprehensive treatment of these function
spaces and their history. We also refer the reader to \cite{ysy} for
relationships among Morrey spaces, Campanato
spaces, and Besov--Triebel--Lizorkin spaces,  to \cite{as18,Sa18} for
some new progress of Besov and Triebel--Lizorkin spaces, and
to \cite{b20,b20b,bd17,bbd20} for various characterizations and applications of Besov and
Triebel--Lizorkin spaces associated with operators.

Particularly, fractional Sobolev spaces play an
important major role in many questions
involving partial differential equations on $\rn$.
It is known that a theory of first order Sobolev spaces
on doubling metric spaces has been
established based on both upper gradients \cite{s00,hk98}
and pointwise inequalities \cite{h96};
see \cite{h00,hk00} for a survey on this.
These different approaches result in the same
function class if the underlying space supports
a suitable Poincar\'e inequality \cite{kz08}.
In this article, we further investigate the spaces
introduced by Haj\l asz \cite{h96} (see also
\cite{y036} and Definition~\ref{h-gra} below)
which are defined via pointwise inequalities.

On another hand, as a generalization of $\rn$,
the space of homogeneous type was introduced by
Coifman and Weiss \cite{cw71,cw77}
(see Definition \ref{h-sp} below),
which provides a natural setting for the study
of function spaces and the boundedness
of Calder\'on--Zygmund operators.
Spaces of homogeneous type, with some additional assumptions,
have been extensively investigated  in many
articles. For instance, the \emph{Ahlfors $d$-regular space} is
a special  space of homogeneous type
satisfying the following condition: there exists
a positive constant $C$ such that, for any ball
$B(x,r)\subset\cx$ with center $x$ and radius $r\in(0,\diam \cx)$,
$$
C^{-1}r^d\leq \mu(B(x,r))\leq Cr^d,
$$
where here, and thereafter, $\diam \cx:=\sup_{x,\ y\in \cx} d(x,y)$.
Another case is the RD-space (see
\cite{j94,hmy06,hmy08} for instance), which is
a doubling metric measure space satisfying the
following additional \emph{reverse doubling condition}:
there exist positive constants $\widetilde{C}_{(\mu)}\in(0,1]$
and $\kappa\in(0,\omega]$ such that, for any
ball $B(x, r)$ with $r\in(0, \diam \cx/2)$ and
$\lambda\in[1,\diam \cx/(2r))$,
\begin{equation*}
\widetilde{C}_{(\mu)}\lambda^\kappa\mu(B(x, r))\leq\mu(B(x, \lambda r)).
\end{equation*}
Obviously, an RD-space is a generalization of
an  Ahlfors $d$-regular space. We refer
the reader to \cite{yz11} for more equivalent characterizations of RD-spaces.

Besov and Triebel--Lizorkin spaces on spaces
of homogeneous type satisfying some
additional assumptions were also studied.
We refer the reader to
\cite{hy02,hy03, y041,y051,y052} for various
characterizations of Besov and Triebel--Lizorkin
spaces on Ahlfors $d$-regular spaces,
and to \cite{y033,y034, y035,y042} for some applications.
We also refer the reader to
\cite{hmy08,my09,yz11} for
various characterizations of these Besov and Triebel--Lizorkin spaces on RD-spaces.
Besides, Koskela et al. \cite{kyz10,kyz11} introduced
the Haj\l asz--Besov and Haj\l asz--Triebel--Lizorkin spaces on RD-spaces.
We refer the reader to \cite{gpy13,hit16,hph17,ht16,ht14b} for
various characterizations and applications of Haj\l asz--Besov and
Haj\l asz--Triebel--Lizorkin on a metric measure space
satisfying the doubling property.

Recently, using the wavelet reproducing formulae in \cite{hlw},
Han et al. \cite{hhhlp20} introduced Besov and
Triebel--Lizorkin spaces on spaces of
homogeneous type and established several  embedding theorems.
On the other hand, Wang et al. \cite{whhy} also
introduced Besov and Triebel--Lizorkin spaces
on spaces of homogenous type,
based on the Calder\'{o}n reproducing formulae established
in \cite{hlyy}, and established the boundedness of
Calder\'on--Zygmund operators on these spaces as
an application. Later, He et al. \cite{hwyy20} obtained  characterizations of
Besov and Triebel--Lizorkin spaces via
wavelets, molecules, Lusin area functions,
and Littlewood--Paley $g_\lambda^\ast$-functions.
Besides, He et al. \cite{hwyy20} also showed that those two kinds of Besov
and Triebel--Lizorkin spaces studied,
respectively, in \cite{hhhlp20} and \cite{whhy} coincide.
Then Wang et al. \cite{whyy} established the  difference
characterization of Besov and Triebel--Lizorkin
spaces on spaces of homogenous type.

To complete the theory of Besov and Triebel--Lizorkin
spaces on spaces of homogenous type,
it is a natural question whether or not we can
also establish a pointwise characterization of
Besov and Triebel--Lizorkin spaces on space
of homogenous type. The main target of this
article is to give an affirmative answer to this question.

The organization of the remainder of this article is as follows.

In Section \ref{s1}, we first recall the notions of homogenous
Besov and Triebel--Lizorkin spaces on spaces of
homogeneous type introduced in \cite{whhy},
and then introduce Haj\l asz--Sobolev spaces
and Haj\l asz--Besov--Triebel--Lizorkin
spaces on spaces of homogeneous type.  Then
we state the main result of this article; see Theorem \ref{mr} below.

In Section \ref{s2}, we first introduce the  homogenous grand
Besov and Triebel--Lizorkin spaces on spaces of homogeneous type.
Then we investigate the relation among
homogeneous grand Besov and
Triebel--Lizorkin spaces and homogeneous Besov
and Triebel--Lizorkin spaces; see Theorem \ref{btl-gbtl} below.
Later, we establish the equivalence between homogenous
Haj\l asz--Besov--Triebel--Lizorkin spaces and  homogeneous Besov
and Triebel--Lizorkin spaces; see Theorem \ref{hbtl-gbtl} below.
To this end, we first establish a Poincar\'e type
inequality (see Lemma \ref{pc} below). It should
be mentioned that, in the proof of Lemma \ref{pc},
the constant $A_0$ appearing in the quasi-
triangle inequality (see Definition \ref{qa-d}
below) also brings some difficulty.
That is why we need additional restrictions on
parameters involved therein. Besides, all the
proofs in Section \ref{s2} get rid of the
dependence on the reverse doubling assumption.

In Section \ref{s4}, we establish the equivalence between
inhomogeneous Haj\l asz--Besov--Triebel--Lizorkin spaces
and  inhomogeneous Besov--Triebel--Lizorkin spaces
(see Theorem \ref{imr} below). To this end, we first
establish a new characterization of inhomogeneous
Besov--Triebel--Lizorkin spaces (see Theorem \ref{h-btl} below).

Finally, let us make some conventions on notation.
For any given $p\in(0,\infty]$, the \emph{Lebesgue space}
$L^p(\cx)$ is defined by setting, when $p\in(0,\infty)$,
$$
L^p(\cx):=\left\{f \ \text{is measurable on} \ \cx :\
\|f\|_{\lp}:=\left[\int_{\cx}|f(x)|^p\,d\mu(x)\right]^{1/p}<\infty\right\},
$$
and
$$
L^\infty(\cx):=\left\{f \ \text{is measurable on} \ \cx :\  \|f\|_{L^\infty(\cx)}:=
\displaystyle{\esssup_{x\in\cx}|f(x)|<\infty}\right\}.
$$
Throughout this article, we use $A_0$ to denote
the positive constant appearing in the \emph{quasi-triangle inequality}
of $d$ (see Definition \ref{qa-d} below), the parameter $\omega$
to denote the \emph{upper dimension} in Definition~\ref{h-sp}
[see \eqref{eq-doub} below], and $\eta$ to denote
the smoothness index of the exp-ATI in Definition~\ref{exp-ati}
below. Moreover, $\delta$ is a small positive number, for instance,
$\delta\leq(2A_0)^{-10}$, coming from the construction of the
dyadic cubes on $\cx$ (see Theorem \ref{exp-ati}). For any given $p\in[1,\infty]$,
we use $p'$ to denote its conjugate index, that is, $1/p + 1/p' = 1$.
For any $r\in \rr$, $r_+$ is defined by setting
$r_+:=\max\{0,r\}$. For any $a,\ b\in\rr$,
let $a\wedge b:=\min\{a,b\}$ and $a\vee b :=\max\{a,b\}$.
The symbol $C$ denotes a positive constant which is
independent of the main parameters involved, but may vary
from line to line. We use $C_{(\alpha,\beta,\dots)}$
to denote a positive constant depending on the indicated
parameters $\alpha,\ \beta,\ \dots$.
The symbol $A\lesssim B$ means that $A\leq CB$ for some
positive constant $C$, while $A\sim B$ means $A\lesssim B\lesssim A$.
If $f\le Cg$ and $g=h$
or $g\le h$, we then write $f\ls g\sim h$ or $f\ls g\ls h$,
\emph{rather than} $f\ls g=h$ or $f\ls g\le h$.
The set of positive integers is denoted by $\nn$, namely, $\nn=\{1,2,\dots\}$,
and the set $\zz_+:=\{0,1,2,\dots\}$.
For any $r\in(0,\infty)$ and $x,\ y\in\cx$ with $x\neq y$, define $V(x,y):=\mu(B(x,d(x,y)))$
and $V_r(x):=\mu(B(x,r))$. For any  $\beta,\ \gamma\in(0,\eta)$ and
$s\in (-(\beta\wedge\gamma),\beta\wedge\gamma)$,
we always let
\begin{equation}\label{pseta}
p(s,\beta\wedge\gamma):=\max\left\{\frac{\omega}{\omega+
(\beta\wedge\gamma)},\frac{\omega}{\omega+(\beta\wedge\gamma)+s}\right\},
\end{equation}
where $\omega$ and $\eta$ are, respectively, as in \eqref{eq-doub}
and Definition \ref{exp-ati}. The operator $M$ always denotes
the \emph{central Hardy--Littlewood maximal
operator} which is defined by setting, for any
locally integrable function $f$ on $\cx$ and any $x\in\cx$,
\begin{equation}\label{m}
M(f)(x):=\sup_{r\in(0,\infty)}\frac{1}{\mu(B(x,r))}\int_{B(x,r)}|f(y)|\,d\mu(y).
\end{equation}
For any set $E\subset\cx$, we use
$\mathbf 1_E$ to denote its characteristic function and,
for any set $J$, we use $\#J$ to denote its
\emph{cardinality}.

\section{Besov and Triebel--Lizorkin spaces
and Haj\l asz--Sobolev spaces on spaces of
homogeneous type}\label{s1}

In this section, we recall the notions of Besov
and Triebel--Lizorkin spaces on spaces of homogeneous type and introduce
Haj\l asz--Sobolev spaces on spaces of homogeneous type.
Let us begin with the notion of quasi-metric spaces.

\begin{definition}\label{qa-d}
A \emph{quasi-metric space} $(\cx, d)$ is a non-empty set $\cx$
equipped with a \emph{quasi-metric} $d$,
namely, a non-negative function defined on $\cx\times\cx$
satisfying that, for any $x,\ y,\ z \in \cx$,
\begin{enumerate}
\item[{\rm(i)}] $d(x,y)=0$ if and only if $x=y$;
\item[{\rm(ii)}] $d(x,y)=d(y,x)$;
\item[{\rm(iii)}] there exists a constant $A_0 \in [1, \infty)$,
independent of $x$, $y$, and $z$, such that
$$d(x,z)\leq A_0[d(x,y)+d(y,z)].$$
\end{enumerate}
\end{definition}

The \emph{ball} $B$ of $\cx$, centered at $x_0 \in \cx$
with radius $r\in(0, \infty)$, is defined by setting
$$
B:=B(x_0,r):=\{x\in\cx: d(x,x_0)<r\}.
$$
For any ball $B$ and $\tau\in(0,\infty)$, we denote by $\tau B$
the ball with the same center as that of $B$ but
of radius $\tau$ times that of $B$.

\begin{definition}\label{h-sp}
Let $(\cx,d)$ be a quasi-metric space and
$\mu$ a non-negative measure on $\cx$.
The triple $(\cx,d,\mu)$ is called
a \emph{space of homogeneous type} if $\mu$ satisfies the
following doubling condition: there exists a
positive constant $C\in[1,\infty)$ such that,
for any ball $B \subset \cx$,
\begin{equation}\label{eq-doub2}
0<\mu(2B)\leq C\mu(B)<\infty.
\end{equation}
\end{definition}
Let
$C_{(\mu)}:=\sup_{B\subset \cx}\mu(2B)/\mu(B)$.
Then it is easy to show that  $C_{(\mu)}$ is the smallest positive constant such that \eqref{eq-doub2} holds true.
The above doubling condition implies that,
for any ball $B$ and any $\lambda\in[1,\infty)$,
\begin{equation}\label{eq-doub}
\mu(\lambda B)\leq C_{(\mu)}\lambda^\omega\mu(B),
\end{equation}
where $\omega:= \log_2 C_{(\mu)}$ is called the \emph{upper dimension} of $\cx$.
Note that $\omega\in(0,\fz)$ (see, for instance, \cite[p.\,72]{rm15}).
If $A_0 = 1$, then
$(\cx,d,\mu)$ is called a \emph{metric measure space of homogeneous type}
or, simply, a \emph{doubling metric measure space}.

Without loss of generality, we may make the following assumptions
on $(\cx,d,\mu)$. For any point $x\in \cx$, we assume that
the balls $\{B(x,r)\}_{r\in(0,\fz)}$ form a basis of open neighborhoods
of $x$. Moreover, we suppose that $\mu$ is \emph{Borel regular}
which means that open sets are measurable and every set
$A\subset \cx$ is contained in a Borel set $E$ satisfying that $\mu(A)=\mu(E)$.
We also assume that $\mu(B(x,r))\in (0,\fz)$ and
$\mu(\{x\})=0$ for any given $x\in \cx$ and $r\in(0,\fz)$.

Now, we recall the notion of test functions and distributions on $\cx$,
whose following versions were originally introduced in
\cite{hmy08} (see also \cite{hmy06}).

\begin{definition}[test functions]
Let $x_1\in\cx$, $r\in(0,\infty)$, $\beta \in (0,1]$,
and $\gamma \in (0,\infty)$. For any $x\in\cx$, define
\begin{equation}\label{decay}
D_\gamma(x_1,x;r):=\frac{1}{V_r(x_1)+V(x_1,x)}\left[\frac{r}{r+d(x_1,x)}\right]^\gamma.
\end{equation}
A measurable function $f$ on $\cx$ is
called a \emph{test function of type $(x_1, r, \beta, \gamma)$}
if there exists a positive constant $C$ such that
\begin{enumerate}
\item[{\rm(i)}] for any $x\in\cx$,
\begin{equation}\label{t-size}
|f(x)|\leq CD_\gamma(x_1,x;r);
\end{equation}
\item[{\rm(ii)}] for any $x,\ y \in \cx$ satisfying $d(x, y)\leq(2A_0)^{-1}[r + d(x_1, x)]$,
\begin{equation}\label{t-reg}
|f(x)-f(y)|\leq C\left[\frac{d(x,y)}{r+d(x_1,x)}\right]^\beta
D_\gamma(x_1,x;r).
\end{equation}
\end{enumerate}
The set of all test functions of type $(x_1, r, \beta, \gamma)$
is denoted by $\cg(x_1, r, \beta,\gamma)$.
For any  $f\in \cg(x_1, r, \beta, \gamma)$,
its norm $\|f\|_{\cg(x_1, r, \beta, \gamma)}$ in $\cg(x_1, r, \beta, \gamma)$ is defined by
setting
$$
\|f\|_{\cg(x_1, r, \beta, \gamma)}:=\inf\{C\in(0,\infty):\,
\text{\eqref{t-size} and \eqref{t-reg} hold true}\}.
$$
Its subspace $\mathring{\cg}(x_1, r, \beta, \gamma)$ is defined by setting
$$
\mathring{\cg}(x_1, r, \beta, \gamma):=\left\{f\in\cg(x_1, r, \beta, \gamma):
\int_{\cx}f(x)\,d\mu(x)=0\right\}
$$
and is equipped with the norm $\|\cdot\|_{\mathring{\cg}(x_1, r, \beta, \gamma)}
:=\|\cdot\|_{\cg(x_1, r, \beta,\gamma)}$.
\end{definition}

Note that, for any fixed $x_1,\ x_2\in\cx$ and $r_1,\ r_2 \in(0,\infty)$,
$\cg(x_1,r_1,\beta,\gamma) = \cg(x_2,r_2,\beta,\gamma)$
and $\mathring\cg(x_1,r_1,\beta,\gamma)
=\mathring\cg(x_2,r_2,\beta,\gamma)$
with equivalent norms, but the positive equivalence  constants
may depend on $x_1$, $x_2$, $r_1$, and $r_2$.
Thus, for fixed $x_0\in\cx$ and $r_0\in(0,\infty)$,
we may denote $\cg(x_0, r_0, \beta, \gamma)$ and
$\mathring{\cg}(x_0, r_0, \beta, \gamma)$ simply,
respectively, by $\cg(\beta,\gamma)$ and $\mathring{\cg}(\beta,\gamma)$.
Usually, the spaces $\cg(\beta,\gamma)$ and $\mathring\cg(\beta,\gamma)$ are called
the \emph{spaces of test functions} on $\cx$.

Fix $\epsilon\in (0, 1]$ and $\beta,\ \gamma \in (0, \epsilon]$.
Let $\cg^\epsilon_0(\beta,\gamma)$ [resp.,
$\mathring{\cg}^\epsilon_0(\beta, \gamma)$] be the completion
of the set $\cg(\epsilon, \epsilon)$ [resp.,
$\mathring{\cg}(\epsilon, \epsilon)$] in $\cg(\beta,\gamma)$
[resp., $\mathring{\cg}(\bz,\gz)$]. Furthermore,
the norm of $\cg^\epsilon_0(\beta,\gamma)$
[resp., $\mathring{\cg}^\epsilon_0(\beta, \gamma)$] is defined
by setting $\|\cdot\|_{\cg^\epsilon_0(\beta,\gamma)}:=\|\cdot\|_{\cg(\beta,\gamma)}$ [resp.,
$\|\cdot\|_{\mathring{\cg}^\epsilon_0(\beta,\gamma)}:=\|\cdot\|_{\cg(\beta,\gamma)}$]. The dual space
$(\cg^\epsilon_0(\beta,\gamma))'$ [resp., $(\mathring{\cg}^\epsilon_0(\beta,\gamma))'$]
is defined to be the
set of all continuous linear functionals from
$\cg^\epsilon_0(\beta,\gamma)$ [resp.,
$\mathring{\cg}^\epsilon_0(\beta,\gamma)$] to $\mathbb{C}$,
equipped with the weak-$\ast$ topology. The
spaces $(\cg^\epsilon_0(\beta,\gamma))'$
and $(\mathring{\cg}^\epsilon_0(\beta,\gamma))'$ are called the
\emph{spaces of distributions} on $\cx$.

The  following lemma, which comes from \cite[Theorem 2.2]{hk},
establishes the dyadic cube system of
$(\cx,d,\mu)$.

\begin{lemma}\label{2-cube}
Let constants $0 < c_0\leq C_0 <\infty$ and $\delta\in(0, 1)$
be such that $12A_0^3C_0\delta\leq c_0$.
Assume that a set of points, $\{ z_\alpha^k : k \in\zz ,
\alpha \in \ca_k \} \subset \cx$ with $\ca_k$ for
any $k \in\zz$ being a set of indices, has the following properties: for any $k \in\zz$,
$$d(z^k_\alpha, z^k_\beta)\geq c_0\delta^k\quad
\text{if}\quad\alpha\neq\beta,\quad
\text{and}\quad\min_{\alpha\in\ca_k}d(x,z_\alpha^k)<C_0\delta^k\quad\text{for any}\quad x\in\cx.$$
Then there exists a family of sets,
$\{ Q_\alpha^k : k \in\zz , \alpha \in \ca_k \}$, satisfying
\begin{enumerate}
\item[{\rm(i)}] for any $k\in\zz$, $\bigcup_{\alpha\in\ca_k}
Q_\alpha^k=\cx$ and $\{ Q_\alpha^k : \alpha \in \ca_k \}$ consists of mutually disjoint sets;
\item[{\rm(ii)}] if $l,\ k\in\zz$ and $k\leq l$, then,
for any $\alpha\in\ca_k$ and $\beta\in\ca_l$,
either $Q_\beta^l\subset Q_\alpha^k$
or $Q_\beta^l\cap Q_\alpha^k=\emptyset$;
\item[{\rm(iii)}] for any $k\in\zz$ and $\alpha\in\ca_k$,
$B(z^k_\alpha, (3A_0^2)^{-1}c_0\delta^k)\subset Q_\alpha^k
\subset B(z^k_\alpha, 2A_0C_0\delta^k)$.
\end{enumerate}
\end{lemma}

Throughout this article, for any $k\in\zz$, define
$$
\cg_k:=\ca_{k+1}\setminus\ca_k\quad\text{and}\quad\cy^k:=\left\{z^{k+1}_\alpha\right\}_{\alpha\in\cg_k}
=: \left\{y_\alpha^k\right\}_{\alpha\in\cg_k}
$$
and, for any $x \in \cx$, define
$$
d(x,\cy^k):=\inf_{y\in\cy^k}d(x,y)\quad\text{and}\quad
V_{\delta^k}(x):=\mu(B(x,\delta^k)).
$$

Now, we recall the notion of approximations of
the identity with exponential decay from \cite{hlyy}.

\begin{definition}\label{exp-ati}
A sequence $\{Q_k\}_{k\in\zz}$ of bounded linear
integral operators on $L^2(\cx)$ is called an
\emph{approximation of the identity with exponential
decay} (for short, exp-ATI) if there exist constants $C$,
$\nu\in(0, \infty)$, $a\in(0, 1]$, and $\eta\in(0,1)$ such that,
for any $k \in\zz$, the kernel of the operator $Q_k$, a
function on $\cx \times \cx$ , which is still denoted by $Q_k$,
satisfies the following conditions:
\begin{enumerate}
\item[{\rm(i)}] (the \emph{identity condition}) $\sum_{k=-\infty}^\infty Q_k=I$
in $L^2(\cx)$, where $I$ denotes the \emph{identity operator} on $L^2(\cx)$;
\item[{\rm(ii)}] (the \emph{size condition}) for any $x,\ y\in\cx$,
\begin{align*}
|Q_k(x,y)|&\leq C \frac{1}{\sqrt{V_{\delta^k}(x)V_{\delta^k}(y)}}H_k(x,y),
\end{align*}
where here, and thereafter,
$$
H_k(x,y):=\exp\left\{-\nu\left[\frac{d(x,y)}{\delta^k}\right]^a\right\}
\exp\left\{-\nu\left[\frac{\max\{d(x,\cy^k),d(y,\cy^k)\}}{\delta^k}\right]^a\right\};
$$
\item[{\rm(iii)}] (the \emph{regularity condition}) for any
$x,\ x',\ y\in\cx$ with $d(x, x')\leq\delta^k$,
\begin{align*}
|Q_k(x,y)-Q_k(x',y)|+|Q_k(y,x)-Q_k(y,x')|
\leq C\left[\frac{d(x,x')}{\delta^k}\right]^\eta
\frac{1}{\sqrt{V_{\delta^k}(x)V_{\delta^k}(y)}}H_k(x,y);
\end{align*}
\item[{\rm(iv)}] (the \emph{second difference regularity condition}) for any
$x,\ x',\ y,\ y'\in\cx$ with $d(x, x')\leq\delta^k$ and $d(y, y')\leq\delta^k$,
\begin{align*}
&|[Q_k(x,y)-Q_k(x',y)]-[Q_k(x,y')-Q_k(x',y')]|\\
&\quad\leq C\left[\frac{d(x,x')}{\delta^k}\right]^\eta\left[\frac{d(y,y')}{\delta^k}\right]^\eta
\frac{1}{\sqrt{V_{\delta^k}(x)V_{\delta^k}(y)}}H_k(x,y);
\end{align*}
\item[{\rm(v)}] (the \emph{cancellation condition}) for any $x,\ y \in \cx$,
$$\int_\cx Q_k(x,y')\,d\mu(y')=0=\int_\cx Q_k(x',y)\,d\mu(x').$$
\end{enumerate}
\end{definition}

The existence of such an exp-ATI on spaces of homogeneous
type is guaranteed by \cite[Theorem 7.1]{ah13} with $\eta$
same as in \cite[Theorem 3.1]{ah13}
which might be very small (see also \cite[Remark~2.8(i)]{hlyy}).
However, if $d$ is a metric,
then $\eta$ can be taken arbitrarily close to 1 (see \cite[Corollary~6.13]{ht14}).

The following lemma states some basic properties of exp-ATIs.
One can find more details in \cite[Remarks 2.8 and 2.9, and
Proposition 2.10]{hlyy}.

\begin{lemma}\label{pro-qk}
Let $\{Q_k\}_{k\in\zz}$ be an {\rm exp-ATI}
and $\eta\in(0,1)$ as in Definition \ref{exp-ati}.
Then, for any given $\Gamma \in (0,\infty)$,
there exists a positive constant $C$ such that, for any $k\in\zz$, the
kernel $Q_k$ has the following properties:
\begin{enumerate}
\item[{\rm(i)}] for any $x,\ y \in\cx$,
\begin{equation}\label{10.23.3}
|Q_k(x,y)|\leq CD_\Gamma(x,y;\delta^k),
\end{equation}
where $D_\Gamma(x,y;\delta^k)$ is as in \eqref{decay};
\item[{\rm(ii)}] for any $x,\ x',\ y\in\cx$ with $d(x,x')\leq(2A_0)^{-1}[\delta^k+d(x,y)]$,
\begin{align}\label{10.23.4}
|Q_k(x,y)-Q_k(x',y)|+|Q_k(y,x)-Q_k(y,x')|
\leq C\left[\frac{d(x,x')}{\delta^k+d(x,y)}\right]^\eta D_\Gamma(x,y;\delta^k);
\end{align}
\item[{\rm(iii)}] for any $x,\ x',\ y,\ y' \in \cx$
with $d(x,x')\leq(2A_0)^{-2}[\delta^k +d(x, y)]$ and
$d(y, y')\leq(2A_0)^{-2}[\delta^k +d(x, y)]$,
\begin{align*}
&|[Q_k(x,y)-Q_k(x',y)]-[Q_k(x,y')-Q_k(x',y')]|\\
&\quad\leq C\left[\frac{d(x,x')}{\delta^k+d(x,y)}\right]^\eta
\left[\frac{d(y,y')}{\delta^k+d(x,y)}\right]^\eta
D_\Gamma(x,y;\delta^k).
\end{align*}
\end{enumerate}
\end{lemma}

Based on exp-ATIs, we now recall the notions
of Besov and Triebel--Lizorkin spaces on spaces of homogeneous type;
see \cite[Definitions 3.1 and 5.1]{whhy}.

\begin{definition}\label{h}
Let $\beta,\ \gamma \in (0, \eta)$ with $\eta$ as in
Definition \ref{exp-ati}, and
$s\in(-(\beta\wedge\gamma), \beta\wedge\gamma)$.
Let $\{Q_k\}_{k\in\zz}$ be an exp-ATI.
\begin{enumerate}
\item[\rm{(i)}] Let $p\in(p(s,\beta\wedge\gamma),\infty]$,
with $p(s,\beta\wedge\gamma)$ as in \eqref{pseta}, and $q \in (0,\infty]$.
The \emph{homogenous Besov space} $\hb$ is defined by setting
$$
\hb := \left\{f  \in\lf(\cgg\r)' :\  \|f\|_{\hb}<\infty\right\},
$$
where, for any $f\in(\cgg)'$,
$$\|f\|_{\hb}:=\left[\sum_{k=-\fz}^\fz
\delta^{-ksq}\|Q_k(f)\|_{\lp}^q\right]^{1/q}$$
with usual modifications made when $q=\infty$.

\item[\rm{(ii)}] Let $p\in(p(s,\beta\wedge\gamma),\infty)$
and $q \in (p(s,\beta\wedge\gamma),\infty]$.
The \emph{homogenous Triebel--Lizorkin space} $\hf$ is defined by setting
$$
\hf := \left\{f  \in \lf(\cgg\r)' :\  \|f\|_{\hf}<\infty\right\},
$$
where, for any $f\in(\cgg)'$,
$$\|f\|_{\hf}:=\left\|\left[\sum_{k=-\fz}^\fz
\delta^{-ksq}|Q_k(f)|^q\right]^{1/q}\right\|_{\lp}$$
with usual modification made when $q=\infty$.
\end{enumerate}
\end{definition}

The following definition introduces the notion of Triebel--Lizorkin
spaces with $p=\infty$; see \cite[Definition~5.1]{whhy}.

\begin{definition}\label{hfi}
Let $\beta,\ \gamma \in (0, \eta)$,
$s\in(-(\beta\wedge\gamma),\beta\wedge\gamma)$,
and $q\in (p(s,\beta\wedge\gamma),\infty]$ with
$\eta$ as in Definition \ref{exp-ati} and $p(s,\beta\wedge\gamma)$ as in
\eqref{pseta}. Let $\{Q_k\}_{k\in\zz}$ be an exp-ATI.
For any $k\in\zz$ and $\alpha\in \ca_k$,
let $Q_\az^k$ be as in Lemma \ref{2-cube}.
Then the \emph{homogeneous Triebel--Lizorkin space $\hfi$} is defined by setting
\begin{align*}
\hfi := \left\{f  \in \left(\cgg\right)' :\  \|f\|_{\hfi}<\infty\right\},
\end{align*}
where, for any $f  \in (\cgg)'$,
$$\|f\|_{\hfi}:=\sup_{l \in \zz}
\sup_{\alpha\in\ca_l}\left[\frac{1}{\mu(Q_\alpha^l)}\int_{Q_\alpha^l}\sum_{k=l}^\infty\delta^{-ksq}
|Q_k(f)(x)|^q\,d\mu(x)\right]^{1/q}$$
with the usual modification made when $q=\infty$.
\end{definition}

\begin{remark}\label{2.8x}
\begin{enumerate}
\item[{\rm(i)}] In Definition \ref{exp-ati}, we need $\diam \cx=\fz$
to guarantee (v). Observe that it was shown in
\cite[Lemma 5.1]{ny97} (see also \cite[Lemma 8.1]{ah13})
that $\diam \cx=\fz$ implies $\mu(\cx)=\fz$. Therefore,
$\diam \cx=\fz$ if and only if $\mu(\cx)=\fz$ under
the assumptions of this article. Due to this, we always assume
that $\mu(\cx)=\infty$ in Sections 2 and 3.
\item[{\rm(ii)}] In \cite{whhy}, Wang et al.  proved
that $\hb$ and $\hf$ in Definition \ref{h} are independent of the choices of
$\bz$ and $\gz$ as in in Definition \ref{h}, and exp-ATIs
(see \cite[Propositions 3.13 and 3.16]{whhy} for more details).
Besides, it was also shown that $\hfi$ in
Definition \ref{hfi} is independent of the choices of
$\bz$ and $\gz$, and exp-ATIs (see \cite[Propositions 5.4 and 5.5]{whhy}
for more details).
\end{enumerate}
\end{remark}

Now, we introduce the notions of $s$-gradients
and $s$-Haj\l asz gradients on spaces of homogenous
type (see, for instance, \cite[Definition 1.1 and (2.1)]{kyz11}).

\begin{definition}\label{h-gra}
Let $s\in (0,\infty)$ and $u$ be a measurable function on $\cx$.
\begin{enumerate}
\item[{\rm(i)}] A nonnegative function $g$ is called an \emph{$s$-gradient} of $u$
if there exists a set $E\subset \cx$ with $\mu(E)=0$ such that,
for any $x,\ y\in \cx\setminus E$,
\begin{equation*}
|u(x)-u(y)|\leq [d(x,y)]^s[g(x)+g(y)].
\end{equation*}
Denote by $\cd^s(u)$ the collection of all $s$-gradients of $u$.
\item[{\rm(ii)}] A sequence of nonnegative functions,
$\{g_k\}_{k\in\zz}$, is called an \emph{$s$-Haj\l asz gradient} of $u$
if there exists a set $E\subset \cx$ with $\mu(E)=0$ such that,
for any $k\in\zz$ and $x,\ y\in \cx\setminus E$ with $\delta^{k+1}\leq d(x,y)<\delta^k$,
\begin{equation*}
|u(x)-u(y)|\leq [d(x,y)]^s[g_k(x)+g_k(y)].
\end{equation*}
Denote by $\dd^s(u)$ the collection of all $s$-Haj\l asz gradients of $u$.
\end{enumerate}
\end{definition}

Next, we introduce the  notions of homogeneous
Haj\l asz--Sobolev spaces, Haj\l asz--Triebel--Lizorkin spaces,
and Haj\l asz--Besov spaces (see, for instance, \cite[Definitions 1.2 and 2.1]{kyz11}).

\begin{definition}\label{h-s-btl}
Let $s\in(0,\infty)$.
\begin{enumerate}
\item[{\rm(i)}] Let $p\in(0,\infty)$.
The \emph{homogeneous Haj\l asz--Sobolev space}
$\dot{M}^{s,p}(\cx)$ is defined to be the set of all measurable
functions $u$ on $\cx$ such that
$$\|u\|_{\dot{M}^{s,p}(\cx)}:=\inf_{g\in \cd^s(u)} \|g\|_{L^p(\cx)}<\infty.$$
\item[{\rm(ii)}] Let $p\in(0,\infty)$ and $q\in(0,\infty]$.
The \emph{homogeneous Haj\l asz--Triebel--Lizorkin space}
$\dot{M}^{s}_{p,q}(\cx)$ is defined to be the set of all
measurable functions $u$ on $\cx$ such that
$$\|u\|_{\dot{M}^s_{p,q}(\cx)}:=\inf_{\{g_k\}_{k\in\zz}\in \dd^s(u)}
\left\|\left(\sum_{k=-\infty}^\infty g_k^q\right)^{\frac{1}{q}}\right\|_{L^p(\cx)}<\infty$$
with the usual modification made when $q=\infty$.
\item[{\rm(iii)}] Let $q\in(0,\infty)$. The \emph{homogeneous Haj\l asz--Triebel--Lizorkin space}
$\dot{M}^{s}_{\infty,q}(\cx)$ is defined to be the set of all
measurable functions $u$ on $\cx$ such that
$$\|u\|_{\dot{M}^s_{\infty,q}(\cx)}:=\inf_{\{g_k\}_{k\in\zz}\in \dd^s(u)}\sup_{k\in\zz}\sup_{x\in\cx}
\left\{\sum_{j=k}^\infty\frac{1}{\mu(B(x,\delta^k))}\int_{B(x,\delta^k)} [g_j(y)]^q\,d\mu(y)\right\}^{\frac{1}{q}}<\infty.$$
\item[{\rm(iv)}] The \emph{homogeneous Haj\l asz--Triebel--Lizorkin space}
$\dot{M}^{s}_{\infty,\infty}(\cx)$ is defined to be the
set of all measurable functions $u$ on $\cx$ such that
$$\|u\|_{\dot{M}^s_{\infty,\infty}(\cx)}:=\inf_{\{g_k\}_{k\in\zz}\in \dd^s(u)}
\left\|\sup_{k\in\zz}g_k\right\|_{L^\infty(\cx)}<\infty.$$
\item[{\rm(v)}] Let $p,\ q\in(0,\infty]$. The \emph{homogeneous Haj\l asz--Besov space}
$\dot{N}^s_{p,q}(\cx)$ is defined to be the set of all
measurable functions $u$ on $\cx$ such that
$$\|u\|_{\dot{N}^s_{p,q}(\cx)}:=\inf_{\{g_k\}_{k\in\zz}\in \dd^s(u)}
\left[\sum_{k=-\infty}^\infty\|g_k\|_{L^p(\cx)}^q\right]^{\frac{1}{q}}<\infty$$
with the usual modification made when $q=\infty$.
\end{enumerate}
\end{definition}

By Definition \ref{h-s-btl}, it is easy to see the following conclusion.
We omit the details.

\begin{proposition}
Let $s\in(0,\infty)$ and $p\in(0,\infty]$. Then $\dot{M}^s_{p,\infty}(\cx)= \dot{M}^{s,p}(\cx)$.
\end{proposition}

Next, we recall the notion of weak lower bounds (see, for instance,
\cite[Definition 1.1]{hhhlp20}, \cite[Definition 4.4]{whyy}, and \cite[(2) or (3)]{agh20}).

\begin{definition}\label{lower}
Let $(\cx, d, \mu)$ be a space of
homogeneous type with upper dimension
$\omega$ as in \eqref{eq-doub}.
The measure $\mu$ is said to have a
\emph{weak lower bound} $Q$ with $Q\in(0,\omega]$ if
there exists a positive constant $C$ and a point $x_0\in\cx$ such that, for
any $r\in[1,\infty)$,
\begin{equation*}
\mu(B(x_0,r))\geq Cr^Q.
\end{equation*}
\end{definition}

\begin{remark}
We point out that, in \cite[Definition 4.4]{whyy},  $\mu$ is said to have a
lower bound $Q$ with $Q\in(0,\omega]$ if
there exists a positive constant $C$ such that, for
any $x\in\cx$ and $r\in(0,\infty)$,
$\mu(B(x,r))\geq Cr^Q$. That is why we call it
the weak lower bound in Definition \ref{lower}.
\end{remark}

As the next result illustrates, it follows from the doubling
property of the measure that the weak lower bound and lower bound
conditions are equivalent when $Q=\omega$, where $\omega$
is as in \eqref{eq-doub}.

\begin{proposition}\label{p2.15}
With $\omega$ as in \eqref{eq-doub}, the measure $\mu$ has a weak
lower bound $Q=\omega$ if and only if it has a lower bound $Q=\omega$.
\end{proposition}

\begin{proof}
Clearly, the lower bound condition implies the weak lower bound condition.
Now, we show the converse. To this end, suppose that the measure $\mu$
has a weak lower bound $Q=\omega$ for
some fixed $x_0\in\cx$. Fix $x\in \cx$ and $r\in(0,\infty)$. Next, choose
$R\in[1,\infty)$ large enough so that $R>r$ and $B(x,r)\subset B(x_0,R)$.
Consider the smallest $k\in\nn$ such that $2A_0R\leq (2A_0)^kr$,
where $A_0\in[1,\infty)$ is the constant in the quasi-triangle inequality.
Note that $k\geq1$ because $r<R$, and hence $(2A_0)^k>1$. Also, this choice of $k$
ensures that $(2A_0)^kr\leq (2A_0)^2R$ which further implies that $B(x_0,R)\subset B(x,(2A_0)^kr)$.
Using this, the weak lower bound $Q=\omega$ for the ball $B(x_0,R)$, the doubling
condition in \eqref{eq-doub2}, and $(2A_0)^kr\leq (2A_0)^2R$, we further conclude that
\begin{equation*}
R^\omega\lesssim\mu(B(x_0,R))\lesssim\mu(B(x,(2A_0)^{k}r))\lesssim (2A_0)^{k\omega}\mu(B(x,r))
\lesssim (2A_0)^{2\omega}\left(\frac{R}{r}\right)^\omega\mu(B(x,r)),
\end{equation*}
from which it follows that $\mu(B(x,r))\gtrsim r^\omega$. Thus, $\mu$ has
a lower bound $Q=\omega$, as wanted. This finishes the proof of Proposition
\ref{p2.15}.
\end{proof}	

Now, we can state our main results of this article.

\begin{theorem}\label{mr}
Let $\beta,\ \gamma\in(0,\eta)$ with $\eta$ as in
Definition \ref{exp-ati}, $s\in(0,\beta\wedge\gamma)$,
$p,\ q$ be as in Definition~\ref{h}, and $\omega$ as in \eqref{eq-doub}.
Assume that the measure $\mu$ of $\cx$ has a weak lower bound $Q=\omega$.
\begin{enumerate}
\item[{\rm(i)}] If $p\in (\omega/(\omega+s),\infty)$ and
$q\in (\omega/(\omega+s),\infty]$, then $\dot{M}^s_{p,q}(\cx)=\hf$.
\item[{\rm(ii)}] If $p\in (\omega/(\omega+s),\infty]$ and $q\in(0,\infty]$,
then $\dot{N}^s_{p,q}(\cx)=\hb$.
\end{enumerate}
\end{theorem}

\section{Relations with homogeneous grand Besov and Triebel--Lizorkin spaces}\label{s2}

Before we prove Theorem \ref{mr}, we need to introduce the notions of
another important spaces, namely, the homogeneous grand Besov
and Triebel--Lizorkin spaces on spaces of homogenous type.

\begin{definition}\label{gbtl}
Let $\eta$ be as in Definition \ref{exp-ati},
$s\in(-\eta,\eta)$, $\beta,\ \gamma\in (0,\eta)$, and $q\in(0,\infty]$.
For any $k\in\zz$ and $x\in\cx$, define
$$\cf_k(x):=\left\{\phi\in\mathring{\cg}_0^\eta(\beta,\gamma):\
\|\phi\|_{\mathring{\cg}(x,\delta^k,\beta,\gamma)}\leq 1\right\}.$$
\begin{enumerate}
\item[\rm{(i)}]
For any given $p\in(0,\infty]$,
the \emph{homogenous grand Besov space}
$\ca\hb$ is defined by setting
$$
\ca\hb := \left\{f  \in\lf(\cgg\r)' :\  \|f\|_{\ca\hb}<\infty\right\},$$
where, for any $f\in (\cgg)'$,
$$\|f\|_{\ca\hb}:=\left[\sum_{k=-\infty}^\infty
\delta^{-ksq}\left\|\sup_{\phi\in\cf_{k}(\cdot)}|\langle f,
\phi\rangle|\right\|_{\lp}^q\right]^{1/q}$$
with usual modification made when $q=\infty$.
\item[\rm{(ii)}]
For any given $p\in(0,\infty)$,
the \emph{homogenous grand Triebel--Lizorkin space}
$\ca\hf$ is defined by setting
$$
\ca\hf := \left\{f  \in \lf(\cgg\r)' :\  \|f\|_{\ca\hf}<\infty\right\}
$$
where, for any $f\in (\cgg)'$,
$$\|f\|_{\ca\hf}:=\left\|\left[\sum_{k=-\infty}^\infty
\delta^{-ksq}\sup_{\phi\in\cf_{k}(\cdot)}|\langle f,\phi\rangle|^q\right]^{1/q}\right\|_{\lp}$$
with usual modification made when $q=\infty$.
\item[\rm{(iii)}]
The \emph{homogenous grand Triebel--Lizorkin space} $\ca\hfi$ is defined by setting
\begin{align*}
\ca\hfi := \left\{f  \in (\cgg)' :\ {} \|f\|_{\ca\hfi}<\infty\right\},
\end{align*}
where, for any $f\in (\cgg)'$,
$$\|f\|_{\ca\hfi}:=\sup_{l \in \zz}
\sup_{\alpha\in\ca_l}\left[\frac{1}{\mu(Q_\alpha^l)}
\int_{Q_\alpha^l}\sum_{k=l}^\infty\delta^{-ksq}
\sup_{\phi\in\cf_{k}(x)}|\langle f,\phi\rangle|^q\,d\mu(x)\right]^{1/q}$$
with usual modification made when $q=\infty$.
\end{enumerate}
\end{definition}

\begin{remark}\label{q-f}
Let $\{Q_k\}_{k\in\zz}$ be an exp-ATI. By \eqref{10.23.3}
and \eqref{10.23.4}, it is easy to  see that,
for any $k\in \zz$ and $x\in\cx$, $Q_k(x,\cdot)\in\cf_k(x)$.
\end{remark}

Now, we establish the relationship between
homogeneous grand Besov
and Triebel--Lizorkin spaces and homogeneous Besov
and Triebel--Lizorkin spaces.

\begin{theorem}\label{btl-gbtl}
Let $\beta,\ \gamma \in (0, \eta)$ with $\eta$ as in
Definition \ref{exp-ati}, and $s\in(-(\beta\wedge\gamma), \beta\wedge\gamma)$.
\begin{enumerate}
\item[{\rm(i)}] If $p$ and $q$ are as in Definition \ref{h}(ii), then $\hf=\ca\hf$.
\item[{\rm(ii)}] If $p$ and $q$ are as in Definition \ref{h}(i), then $\hb=\ca\hb$.
\end{enumerate}
\end{theorem}

To prove Theorem \ref{btl-gbtl}, we need several lemmas.
Let us begin with recalling the following very basic inequality.

\begin{lemma}
For any $\theta\in (0,1]$ and $\{a_j\}_{j\in\nn} \subset \mathbb{C}$,
it holds true that
\begin{equation}\label{r}
\left(\sum_{j=1}^\infty|a_j|\right)^\theta\leq\sum_{j=1}^\infty|a_j|^\theta.
\end{equation}
\end{lemma}

The following lemma contains several basic and very useful estimates
related to $d$ and $\mu$ on $\cx$.
One can find the details in \cite[Lemma 2.1]{hmy08} or \cite[Lemma 2.4]{hlyy}.

\begin{lemma}\label{6.15.1}
Let $\beta,\ \gamma\in(0,\infty)$.
\begin{enumerate}
\item[{\rm(i)}] For any $x,\ y\in \cx$ and $r \in (0, \infty)$,
$V(x, y)\sim V (y, x)$ and
$$V_r(x) + V_r(y) + V (x,y) \sim V_r(x) + V (x,y) \sim V_r(y) + V (x,y) \sim \mu(B(x,r + d(x,y)))$$
and, moreover, if $d(x,y)\leq r$, then $V_r(x)\sim V_r(y)$.
Here the positive equivalence  constants are independent of $x$, $y$, and $r$.
\item[{\rm(ii)}] There exists a positive constant $C$ such that,
for any $x_1 \in \cx$ and $r \in (0, \fz)$,
$$
\int_\cx D_\gamma(x_1,y;r)\,d\mu(y)\leq C;
$$
where, $D_\gamma(x,y;r)$ is as in \eqref{decay}.
\end{enumerate}
\end{lemma}

The following homogeneous discrete Calder\'on
reproducing formula was obtained in \cite[Theorem 5.11]{hlyy}.
Let $j_0\in\nn$ be sufficiently large such that $\delta^{j_0}
\leq (2A_0)^{-3}C_0$. Based on Lemma
\ref{2-cube}, for any $k\in\zz$ and $\alpha\in\ca_k$, let
$$
{\mathfrak N}(k,\alpha):=\{\tau\in\ca_{k+j_0}:\  Q_\tau^{k+j_0}\subset Q_\alpha^k\}
$$
and $N(k,\alpha):=\#{\mathfrak N}(k,\alpha)$. From Lemma \ref{2-cube}, it follows that
$N(k,\alpha) \lesssim \delta^{-j_0\omega}$ and $\bigcup_{\tau\in{\mathfrak N}(k,\alpha)}
Q_\tau^{k+j_0}= Q_\alpha^k$. We
rearrange the set $\{Q_\tau^{k+j_0}:\tau\in{\mathfrak N}(k,\alpha)\}$
as $\{\qa\}_{m=1}^{N(k,\alpha)}$. Also, denote by
$\ya$ an arbitrary point in $\qa$ and $z_\alpha^{k,m}$ the ``center" of $\qa$.

\begin{lemma}\label{crf}
Let $\{Q_k\}_{k=-\infty}^\infty$ be an {\rm exp-ATI}
and $\beta,\ \gamma \in (0, \eta)$ with $\eta$ as in Definition \ref{exp-ati}.
For any $k\in\zz$, $\alpha\in\ca_k$, and $m\in\{1,\dots,N(k,\alpha)\}$,
suppose that $\ya$ is an arbitrary point in $\qa$.
Then there exists a sequence $\{\widetilde{Q}_k\}_{k=-\infty}^\infty$ of
bounded linear integral operators on $L^2(\mathcal{X})$ such that,
for any $f \in (\cgg)'$,
$$f(\cdot) = \sum_{k=-\infty}^\infty\sum_{\alpha \in \ca_k}
\sum_{m=1}^{N(k,\alpha)}\mu\left(\qa\right)\widetilde{Q}_k(\cdot,\ya)Q_kf
\left(\ya\right).$$
in $(\cgg)'$.
Moreover, there exists a positive constant $C$,
independent of the choices of both
$\ya$, with $k\in\zz,\ \alpha\in\ca_k$, and $m\in\{1,\dots,N(k,\alpha)\}$,
and $f$, such that,
for any $k\in\zz$,
the kernel of $\widetilde{Q}_k$
satisfies
\begin{enumerate}
\item[{\rm(i)}] for any $x,\ y \in\cx$,
\begin{equation}\label{4.23x}
\left|\widetilde{Q}_k(x,y)\right|\leq CD_\gamma(x,y;\delta^k),
\end{equation}
where $D_\gamma(x,y;\delta^k)$ is as in \eqref{decay};
\item[{\rm(ii)}] for any $x,\ x',\ y\in\cx$ with $d(x,x')\leq(2A_0)^{-1}[\delta^k+d(x,y)]$,
\begin{align}\label{4.23y}
&\left|\widetilde{Q}_k(x,y)-\widetilde{Q}_k(x',y)\right|\leq
C\left[\frac{d(x,x')}{\delta^k+d(x,y)}\right]^\beta
D_\gamma(x,y;\delta^k);
\end{align}
\item[{\rm(iii)}]for any $x\in\mathcal{X}$,
\begin{equation*}
\int_{\mathcal{X}}\widetilde{Q}_k(x,y)\,d\mu(y)=0=
\int_{\mathcal{X}}\widetilde{Q}_k(y,x)\,d\mu(y).
\end{equation*}
\end{enumerate}
\end{lemma}

We also need the following three lemmas (see,
for instance, \cite[Lemmas 3.5 and 3.6]{whhy}).

\begin{lemma}\label{9.14.1}
Let $\gamma \in (0,\infty)$ and $p \in (\omega/(\omega+\gamma),1]$
with $\omega$ as in \eqref{eq-doub}.
Then there exists a constant $C\in[1,\infty)$ such that,
for any $k,\ k' \in \zz$, $x \in \mathcal{X}$,  and  $\ya \in \qa$ with
$\alpha \in \ca_k$ and $m\in\{1,\dots,N(k,\alpha)\}$,
\begin{align*}
C^{-1}[V_{\delta^{k\wedge k'}}(x)]^{1-p}\le\sum_{\alpha \in \ca_k}
\sum_{m=1}^{N(k,\alpha)}\mu\left(\qa\right)
\left[D_\gamma(x,\ya;\delta^{k\wedge k'})\right]^p
\leq C[V_{\delta^{k\wedge k'}}(x)]^{1-p},
\end{align*}
where $D_\gamma(x,\ya;\delta^{k\wedge k'})$ is as in \eqref{decay}.
\end{lemma}

\begin{lemma}\label{10.18.5}
Let $\gamma \in (0,\infty)$ and $r \in (\omega/(\omega+\gamma),1]$
with $\omega$ as in
\eqref{eq-doub}. Then there exists a positive constant $C$ such that,
for any $k,\ k' \in \zz$, $x \in \mathcal{X}$, and $a_\alpha^{k,m}\in\cc$
and $\ya \in \qa$ with $\alpha \in \ca_k$ and $m\in\{1,\dots,N(k,\alpha)\}$,
\begin{align*}
&\sum_{\alpha \in \ca_k}\sum_{m=1}^{N(k,\alpha)}\mu\left(\qa\right)
D_\gamma(x,\ya;\delta^{k\wedge k'})
|a_\alpha^{k,m}|\\
&\quad\leq C \delta^{[k-(k\wedge k')]\omega(1-1/r)}\left[M\left(\sum_{\alpha \in \ca_k}
\sum_{m=1}^{N(k,\alpha)}|a_\alpha^{k,m}|^r\mathbf 1_{\qa}\right)(x)\right]^{1/r},\notag
\end{align*}
where $D_\gamma(x,\ya;\delta^{k\wedge k'})$ is as in \eqref{decay}
and $M$ as in \eqref{m}.
\end{lemma}

The following lemma is the Fefferman--Stein vector-valued maximal inequality
which was established in \cite[Theorem 1.2]{gly09}.
\begin{lemma}\label{fsvv}
Let $p\in(1,\infty)$, $q\in(1,\infty]$, and $M$ be the Hardy--Littlewood
maximal operator on $\mathcal{X}$ as
in \eqref{m}. Then there exists a positive constant $C$
such that, for any sequence $\{f_j\}_{j\in\zz}$ of
measurable functions on $\mathcal{X}$,
$$
\left\|\left\{\sum_{j\in\zz}[M(f_j)]^q\right\}^{1/q}\right\|_{L^p(\mathcal{X})}
\leq C\left\|\left(\sum_{j\in\zz}|f_j|^q\right)^{1/q}\right\|_{L^p(\mathcal{X})}
$$
with the usual modification made when $q = \infty$.
\end{lemma}

Now, we prove Theorem \ref{btl-gbtl}.

\begin{proof}[Proof of Theorem \ref{btl-gbtl}]
We first show (i).
Assume that $f\in\ca\hf$ and that $\{Q_k\}_{k\in\zz}$ is an exp-ATI.
From Remark \ref{q-f}, we deduce that
$|Q_k(f)|\leq\sup_{\phi\in\cf_{k}(\cdot)}|\langle f,\phi\rangle|$
and hence $\|f\|_{\hf}\leq \|f\|_{\ca\hf}$.

Conversely, assume that $f\in\hf$. By Lemma \ref{crf}, we know that,
for any  $l\in\zz$, $x\in\cx$, and $\phi\in\cf_l(x)$,
$$\langle f,\phi\rangle=\sum_{k=-\infty}^\infty\sum_{\alpha \in \ca_k}
\sum_{m=1}^{N(k,\alpha)}\mu\left(\qa\right)Q_kf
\left(\ya\right)\int_{\cx}\widetilde{Q}_k(z,\ya)\phi(z)\,d\mu(z).$$
Notice that, by an argument similar to that used in the
proof of \cite[Lemma 3.9]{whhy},  we have, for any fixed $\eta'\in(0,\beta\wedge\gamma)$,
$$\left|\int_{\cx}\widetilde{Q}_k(z,\ya)\phi(z)\,d\mu(z)\right|
\lesssim\delta^{|k-l|\eta'}D_\gamma(x,\ya;\delta^{k\wedge l}),$$
where $D_\gamma(x,\ya;\delta^{k\wedge k'})$ is as in \eqref{decay}.
Using this, Lemma \ref{10.18.5}, the arbitrariness of $\ya$,
and choosing $r\in(\omega/(\omega+\gamma), \min\{p,q,1\})$, we obtain
\begin{align*}
|\langle f,\phi\rangle|&\lesssim \sum_{k=-\infty}^\infty\delta^{|k-l|\eta'}
\sum_{\alpha\in\ca_k}\sum_{m=1}^{N(k,\alpha)}
\mu\left(\qa\right)\left|Q_kf\left(\ya\right)\right|D_\gamma(x,\ya;\delta^{k\wedge l})\\
&\lesssim \sum_{k=-\infty}^\infty\delta^{|k-l|\eta'}\delta^{[k-(k\wedge l)]\omega(1-1/r)}
\left[M\left(\sum_{\alpha \in \ca_k}
\sum_{m=1}^{N(k,\alpha)}\left|Q_kf\left(\ya\right)\right|^r\mathbf 1_{\qa}\right)(x)\right]^{1/r}\\
&\lesssim \sum_{k=-\infty}^\infty\delta^{|k-l|\eta'}\delta^{[k-(k\wedge l)]\omega(1-1/r)}
\left[M\left(\left|Q_kf\right|^r\right)(x)\right]^{1/r}.
\end{align*}
By this, we know that
\begin{align*}
\|f\|_{\ca\hf}&\lesssim \left\|\left[\sum_{l=-\infty}^\infty\delta^{(l-k)sq}
\left\{\sum_{k=-\infty}^\infty\delta^{|k-l|\eta'}\delta^{[k-(k\wedge l)]\omega(1-1/r)}\right.\right.\right.\\
&\quad\left.\left.\left.\times\vpz{\sum_{k=-\infty}^\infty}
\left[M\left(\delta^{-ksr}
\left|Q_kf\right|^r\right)\right]^{1/r}\right\}^q\right]^{1/q}\right\|_{L^p(\cx)},
\end{align*}
which, together with the H\"older inequality when $q\in(1,\infty]$,
or \eqref{r} when $q\in(\omega/[\omega+(\beta\wedge\gamma)],1]$, implies that
$$\|f\|_{\ca\hf}\lesssim\left\|\left\{\sum_{k=-\infty}^\infty
\left[M\left(\delta^{-ksr}\left|Q_kf\right|^r\right)\right]^{q/r}\right\}^{1/q}\right\|_{L^p(\cx)}.$$
From this and Lemma \ref{fsvv}, we deduce  that
$$\|f\|_{\ca\hf}\lesssim\|f\|_{\hf}.$$
This finishes the proof of (i).

The proof of (ii) is similar to that of (i) and we omit the details.
\end{proof}

Next, we establish the equivalence between
homogenous
Haj\l asz--Besov spaces and Haj\l asz--Triebel--Lizorkin
spaces, and  homogeneous grand Besov
and Triebel--Lizorkin spaces.

\begin{theorem}\label{hbtl-gbtl}
Let $\beta,\ \gamma\in (0,\eta)$ with $\eta$ as in Definition \ref{exp-ati},
and $s\in(0,\beta\wedge\gamma)$.
Assume that the measure $\mu$ of $\cx$ has a weak lower bound $Q=\omega$.
\begin{enumerate}
\item[{\rm(i)}] If $p\in(\omega/(\omega+s), \infty]$
and $q\in(\omega/(\omega+s), \infty]$, then $\ca\hf=\dot{M}^{s}_{p,q}(\cx)$.
\item[{\rm(ii)}] If  $p\in(\omega/(\omega+s), \infty]$
and $q\in(0, \infty]$, then $\ca\hb=\dot{N}^{s}_{p,q}(\cx)$.
\end{enumerate}

\end{theorem}

To prove Theorem \ref{hbtl-gbtl}, we need several lemmas.
The following lemma  was originally shown in \cite[Theorem 8.7]{h00}
when $s=1$ and $A_0=1$.
When $s\in(0,1)$ and $A_0\in(1,\infty)$, we need more
restrictions on $A_0$ and $\delta$.
We borrow some ideas from the proof of \cite[Theorem 8.7]{h00}.
In what follows, for any measurable set $E\subset \cx$ with $\mu(E)>0$,
let
$$\oint_{E}:=\frac{1}{\mu(E)}\int_{E}.$$

\begin{lemma}\label{pc}
Let $s\in(0,\fz)$, $p\in(0,\omega/s)$, and  $p^\ast:=\frac{\omega p}{\omega-sp}$
with $\omega$ as in \eqref{eq-doub}.
If $A_0\delta^{p/\omega}<1$, then there exists a positive
constant $C$ such that, for any $B_0:=B(x_0,r_0)\subset\cx$
with $x_0\in\cx$ and $r_0\in(0,\infty)$,
$u\in\dot{M}^{s,p}(B(x_0,\delta^{-1}r_0))$, and $g\in\cd^s(u)$, one has $u\in L^{p^\ast}(B_0)$ and
\begin{equation}\label{pc-eq}
\inf_{c\in\rr}\left[\oint_{B_0}|u(y)-c|^{p^\ast}\,d\mu(y)\right]^{\frac{1}{p^\ast}}
\leq Cr_0^s\left\{\oint_{\delta^{-1}B_0}[g(y)]^p\,d\mu(y)\right\}^{1/p}.
\end{equation}
\end{lemma}

\begin{proof}
If $\int_{\delta^{-1}B_0}[g(y)]^p\,d\mu(y)=\infty$,
then \eqref{pc-eq} holds true. If
$$\int_{\delta^{-1}B_0}[g(y)]^p\,d\mu(y)=0,$$
then we know that $g(x)=0$ for almost every $x\in \delta^{-1}B_0$ and
hence there exists a $c\in\rr$ such that $u(x)=c$ for almost every $x\in \delta^{-1}B_0$.
Thus, in this case, \eqref{pc-eq} holds true.

In what follows,  we assume that
$$0<\int_{\delta^{-1}B_0}[g(y)]^p\,d\mu(y)<\infty.$$
Note that this implies $g>0$ almost everywhere in $\cx$.
Moreover, we may also assume that, for every $x\in\delta^{-1}B_0$,
\begin{equation}\label{g-low}
g(x)\geq \delta^{1+1/p}\left\{\oint_{\delta^{-1}B_0}[g(y)]^p\,d\mu(y)\right\}^{1/p},
\end{equation}
as otherwise we may replace $g$ by
$\widetilde{g}(x):=g(x)
+\{\oint_{\delta^{-1}B_0}[g(y)]^p\,d\mu(y)\}^{1/p}$
for any $x\in\delta^{-1}B_0$, because
$$\left\{\oint_{\delta^{-1}B_0}[\widetilde{g}(y)]^p\,d\mu(y)\right\}^{1/p}
\lesssim \left\{\oint_{\delta^{-1}B_0}[g(y)]^p\,d\mu(y)\right\}^{1/p}.$$
For any $k\in\zz$, define
$$E_k:=\left\{x\in\delta^{-1}B_0:\ g(x)\leq \delta^{-k}\right\}.$$
It is easy to see that, for any $k\in\zz$,
$E_{k-1}\subset E_{k}$ and
\begin{equation}\label{ek}
\lim_{k\to\infty}\mu(E_k)=\mu(\delta^{-1}B_0).
\end{equation}
Since $g>0$ almost everywhere in $\cx$, we also have
\begin{equation*}
\mu\left(\delta^{-1}B_0\setminus\bigcup_{k\in\zz}(E_k\setminus E_{k-1})\right)=0,
\end{equation*}
which allows us to write
\begin{equation}\label{g}
\int_{\delta^{-1}B_0}[g(y)]^p\,d\mu(y)\sim
\sum_{k=-\infty}^\infty\delta^{-kp}\mu(E_k\setminus E_{k-1}).
\end{equation}
For any $c\in\rr$, if we let $a_k:=\sup_{y\in B_0\cap E_k}|u(y)-c|$,
then $a_k$ is nondecreasing and
\begin{equation}\label{u}
\int_{B_0} |u(y)-c|^{p^\ast}\,d\mu(y)\leq
\sum_{k=-\infty}^\infty a_k^{p^\ast}\mu(B_0\cap[E_k\setminus E_{k-1}]).
\end{equation}
Note that, if $\mu(\delta^{-1}B_0\setminus E_{k-1})=0$,
then $\mu(E_k\setminus E_{k-1})=0$. Thus, to estimate \eqref{g} and \eqref{u},
we only need to consider $k\in\zz$ such that
$\mu(\delta^{-1}B_0\setminus E_{k-1})>0$, which is always
assumed in what follows. Let
\begin{equation}\label{b}
b:=(4A_0)^{-\omega}\delta^\omega r_0^{-\omega}\mu(\delta^{-1}B_0)
\end{equation}
and
$$
r_k:=2b^{-\frac{1}{\omega}}[\mu(\delta^{-1}B_0\setminus E_{k-1})]^{\frac{1}{\omega}}.
$$
Then we know that $r_k\in(0,\infty)$. Moreover, by the Chebyshev inequality, we know that
\begin{equation}\label{ekc}
\mu(\delta^{-1}B_0\setminus E_k)=\mu(\{x\in\delta^{-1}B_0:\ g(x)>\delta^{-k}\})
\leq\delta^{kp}\int_{\delta^{-1}B_0}[g(y)]^p\,d\mu(y),
\end{equation}
which implies that $\lim_{k\to\infty}r_k=0$. Thus, there exists a $k_0\in\zz$,
which will be determined later, such that, for any $k>k_0$,  we can find an $x_k\in B_0$
satisfying $B(x_k,r_k)\subset \delta^{-1}B_0$, where $r_k\leq\delta^{-1}r_0$.
Observe that, by the doubling condition of $\cx$, we can conclude that, for any $k>k_0$,
$\mu(B(x_k,r_k))\geq br_k^{\omega}$.
Combining this and the definition of $r_k$, we find that
$$\mu(B(x_k,r_k))\geq br_k^{\omega}>\mu(\delta^{-1}B_0
\setminus E_{k-1})=\mu(\delta^{-1}B_0)-\mu(E_{k-1}).$$
From this, we deduce that $B(x_k,r_k)\cap E_{k-1}\neq\emptyset$, that is,
there exists an $x_{k-1}\in B(x_k,r_k)\cap E_{k-1}$.
Now, if $B(x_{k-1},r_{k-1})\subset \delta^{-1}B_0$,
then we can repeat the above procedure to find an $x_{k-2}$
such that $x_{k-2}\in B(x_{k-1},r_{k-1})\cap E_{k-2}$.
As a summary, for any $i\in\{1,\dots,k-k_0+1\}$,
if $B(x_{k-i},r_{k-i})\subset \delta^{-1}B_0$,
then we can find an $x_{k-i-1}$ such that
$x_{k-i-1}\in B(x_{k-i},r_{k-i})\cap E_{k-i-1}$.
We now want to determine $k_0$.
Note that, by \eqref{ekc}, $x_k\in B_0$, and
the assumption that $A_0\delta^{p/\omega}<1$,
we know that, for any $y\in B(x_{k_0},r_{k_0})$,
\begin{align*}
d(y,x_0)&\leq A_0[d(y,x_k)+d(x_k, x_0)]< A_0d(y,x_k)+A_0r_0\\
&\leq A_0^2[d(y,x_{k_0})+d(x_{k_0},x_k)]+A_0r_0\\
&\leq A_0^2r_{k_0}+A_0^3[d(x_{k_0},x_{k_0+1})+d(x_{k_0+1},x_k)]+A_0r_0\\
&\leq A_0^2r_{k_0}+A_0^3r_{k_0+1}+\cdots
+A_0^{k-k_0+1}r_{k-1}+A_0^{k-k_0+1}r_{k}
+A_0r_0\\
&\leq 2A_0^{-k_0+3}b^{-1/\omega}
\left\{\int_{\delta^{-1}B_0}[g(z)]^p\,d\mu(z)\right\}^{1/\omega}
\sum_{i=k_0-1}^{k-1}(A_0\delta^{p/\omega})^i+A_0r_0\\
&\leq A_0^2\delta^{(k_0-1)p/\omega}\frac{2b^{-1/\omega}}{1-A_0\delta^{p/\omega}}
\left\{\int_{\delta^{-1}B_0}[g(z)]^p\,d\mu(z)\right\}^{1/\omega}+A_0r_0.
\end{align*}
If
\begin{equation}\label{koi}
A_0^2\delta^{(k_0-1)p/\omega}\frac{2b^{-1/\omega}}{1-A_0\delta^{p/\omega}}
\left\{\int_{\delta^{-1}B_0}[g(z)]^p\,d\mu(z)\right\}^{1/\omega}+A_0r_0\leq \delta^{-1}r_0,
\end{equation}
then we know that, for any $i\in\{1,\dots,k-k_0\}$,
$B(x_{k-i},r_{k-i})\subset \delta^{-1}B_0$.
Observe that \eqref{koi} is equivalent to
\begin{equation}\label{ko}
\delta^{1-k_0}\geq
\left[\frac{2A_0^2}{(1-A_0\delta^{p/\omega})(\delta^{-1}-A_0)}\right]^{\omega/p}
(br_0^\omega)^{-1/p}\left\{\int_{\delta^{-1}B_0}[g(z)]^p\,d\mu(z)\right\}^{1/p}.
\end{equation}
We claim that, if \eqref{ko} holds true,
then, for any $k\geq k_0$, $r_k\leq \delta^{-1}r$.
Indeed, by the definition of $r_k$, \eqref{ekc}, \eqref{ko}, and
the fact that $\delta$ is very small, we conclude that
\begin{align*}
r_k&\leq 2b^{-\frac{1}{\omega}}\left\{\delta^{(k-1)p}\int_{\delta^{-1}B_0}[g(y)]^p\,d\mu(y)\right\}^\frac{1}{\omega}\\
&\leq 2b^{-\frac{1}{\omega}}\delta^{\frac{(k_0-1)p}{\omega}}
\left\{\int_{\delta^{-1}B_0}[g(y)]^p\,d\mu(y)\right\}^{\frac{1}{\omega}}\\
&\leq 2b^{-\frac{1}{\omega}}
\left[\frac{2A_0^2}{(1-A_0\delta^{p/\omega})(\delta^{-1}-A_0)}
\right]^{-1}b^{\frac{1}{\omega}}r_0 \left\{\int_{\delta^{-1}B_0}[g(y)]^p\,d\mu(y)\right\}^{-\frac{1}{\omega}}\\
&\leq \frac{(1-A_0\delta^{p/\omega})(\delta^{-1}-A_0)}{A_0^2}r_0\\
&\leq \delta^{-1}r_0.
\end{align*}
Thus, the above claim holds true.
Observe that \eqref{g-low} implies $E_k=\emptyset$ for $k\in\zz$ small enough.
By this and \eqref{ek}, we conclude that there exists a $\widetilde{k}_0\in\zz$ such that
\begin{equation}\label{tko}
\mu(E_{\widetilde{k}_0-1})<\delta \mu(\delta^{-1}B_0)
\leq\mu(E_{\widetilde{k}_0}).\end{equation}
From this, we deduce that $E_{\widetilde{k}_0}\neq\emptyset$ and, by \eqref{g-low},
we have, for any $x\in E_{\widetilde{k}_0}$,
$$
\delta^{1+1/p}\left\{\oint_{\delta^{-1}B_0}[g(y)]^p\,d\mu(y)\right\}^{1/p}
\leq g(x)\leq \delta^{-\widetilde{k}_0}.
$$
On the other hand, since $\delta$ is very small,
we may assume that $\delta<1/2$. Then, by
\eqref{tko} and \eqref{ekc}, we know that
$$
\delta\mu(\delta^{-1}B_0)<(1-\delta)\mu(\delta^{-1}B_0)<\mu(\delta^{-1}B_0\setminus E_{\widetilde{k}_0-1})
\leq \delta^{(\widetilde{k}_0-1)p}\int_{\delta^{-1}B_0}[g(y)]^p\,d\mu(y).
$$
Combining the above two estimates, we find that
$$
\delta^{1+1/p}\left\{\oint_{\delta^{-1}B_0}[g(y)]^p\,d\mu(y)\right\}^{1/p}
\leq \delta^{-\widetilde{k}_0}
\leq \delta^{-1-1/p}\left\{\oint_{\delta^{-1}B_0}[g(y)]^p\,d\mu(y)\right\}^{1/p}.
$$
Let $l_0$ be the smallest integer such that
$$\delta^{-l_0}>\max\left\{\delta^{-2-1/p}\left[\frac{2A_0^2}
{(1-A_0\delta^{p/\omega})(\delta^{-1}-A_0)}\right]^{\omega/p},1\right\}
\left[\frac{\mu(\delta^{-1}B_0)}{br_0^\omega}\right]^{1/p}$$
and let $k_0:=\widetilde{k}_0+l_0$.
Then we conclude that \eqref{ko} holds true and
\begin{align*}
\delta^{-k_0}&=\delta^{-1}\delta^{-\widetilde{k}_0}\delta^{-(l_0-1)}\\
&\leq \delta^{-1}\delta^{-1-1/p}
\left\{\oint_{\delta^{-1}B_0}[g(y)]^p\,d\mu(y)\right\}^{1/p}\\
&\quad\times\max\left\{\delta^{-2-1/p}\left[\frac{2A_0^2}
{(1-A_0\delta^{p/\omega})(\delta^{-1}-A_0)}\right]^{\omega/p},1\right\}
\left[\frac{\mu(\delta^{-1}B_0)}{br_0^\omega}\right]^{1/p}\\
&\lesssim (br_0^\omega)^{-1/p}
\left\{\int_{\delta^{-1}B_0}[g(z)]^p\,d\mu(z)\right\}^{1/p},
\end{align*}
which, together with \eqref{ko}, implies that
\begin{equation}\label{e-ko}
\delta^{-k_0}\sim(br_0^\omega)^{-1/p}
\left\{\int_{\delta^{-1}B_0}[g(z)]^p\,d\mu(z)\right\}^{1/p}.
\end{equation}
Now, we estimate $a_k$. We consider two cases on $k$.

{\it Case 1)} $k>k_0$. In this case, it suffices to consider
$k>k_0$ such that $E_k\cap B_0\neq\emptyset$.
For any $x_k\in E_k\cap B_0$, choose
$\{x_{k-1},\dots,x_{k_0}\}$ as above.
Then, by $g\in\cd^s(u)$, the definition of $r_k$, \eqref{ekc}, and $p\in(0,\omega/s)$,
we find that, for any $c\in\rr$,
\begin{align*}
|u(x_k)-c|&\leq\sum_{i=0}^{k-k_0-1}|u(x_{k-i})-u(x_{k-i-1})|+|u(x_{k_0})-c|\\
&\leq \sum_{i=0}^{k-k_0-1}[d(x_{k-i},x_{k-i-1})]^s[g(x_{k-i})+g(x_{k-i-1})]+|u(x_{k_0})-c|\\
&\lesssim \sum_{i=0}^{k-k_0-1}\delta^{-k+i}r_{k-i}^s+|u(x_{k_0})-c|\\
&\lesssim b^{-s/\omega}\left\{\int_{\delta^{-1}B_0}[g(z)]^p\,d\mu(z)\right\}^{s/\omega}
\sum_{i=0}^{k-k_0-1}\delta^{-k+i}\delta^{(k-i-1)ps/\omega}+|u(x_{k_0})-c|\\
&\lesssim b^{-s/\omega}
\left\{\int_{\delta^{-1}B_0}[g(z)]^p\,d\mu(z)\right\}^{s/\omega}
\delta^{-(k-1)(1-ps/\omega)}+|u(x_{k_0})-c|,
\end{align*}
which implies that
$$a_k\lesssim b^{-s/\omega}
\left\{\int_{\delta^{-1}B_0}[g(z)]^p\,d\mu(z)\right\}^{s/\omega}
\delta^{-k(1-ps/\omega)}+\sup_{x\in E_{k_0}}|u(x)-c|.$$
Choose $\widetilde{c}\in\rr$ such that $\essinf_{x\in E_{k_0}}|u(x)-\widetilde{c}|=0$.
Then we can find $\{y_j\}_{j\in\nn}\subset E_{k_0}$
such that $\lim_{j\to\infty}|u(y_j)-\widetilde{c}|=0$.
By $g\in\cd^s(u)$ and the definition of $E_{k_0}$,
we have, for any $x\in E_{k_0}$,
\begin{align}\label{ci}
|u(x)-\widetilde{c}|&=\lim_{j\to\infty}|[u(x)-\widetilde{c}]-[u(y_j)-\widetilde{c}]|
=\lim_{j\to\infty}|u(x)-u(y_j)|\notag\\
&\leq \varlimsup_{j\to\infty}[d(x,y_j)]^s[g(x)+g(y_j)]
\leq 2^{s+1}A_0^sr_0^s\delta^{-k_0-s},
\end{align}
which further implies that, for any $k>k_0$,
\begin{equation}\label{aki}
a_k\lesssim b^{-s/\omega}
\left\{\int_{\delta^{-1}B_0}[g(z)]^p\,d\mu(z)\right\}^{s/\omega}
\delta^{-k(1-ps/\omega)}+r_0^s\delta^{-k_0}.
\end{equation}
{\it Case 2)} $k\leq k_0$. In this case, by \eqref{ci}
and the fact that $E_k$ is increasing, we know that
\begin{equation}\label{akii}
a_k=\sup_{y\in B_0\cap E_k}|u(y)-\widetilde{c}|
\leq \sup_{y\in B_0\cap E_{k_0}}|u(y)-\widetilde{c}|
\leq \sup_{y\in E_{k_0}}|u(y)-\widetilde{c}|\lesssim r_0^s\delta^{-k_0},
\end{equation}
where we let $a_k:=0$ if $B_0\cap E_k=\emptyset$.

From \eqref{aki}, \eqref{akii}, \eqref{u}, \eqref{g},
\eqref{e-ko}, and \eqref{b}, we deduce that
\begin{align*}
\int_{B_0} |u(y)-\widetilde{c}|^{p^\ast}\,d\mu(y)
&\leq\sum_{k=-\infty}^\infty a_k^{p^\ast}\mu(B_0\cap[E_k\setminus E_{k-1}])\\
&\lesssim b^{-sp^\ast/\omega}
\left\{\int_{\delta^{-1}B_0}[g(z)]^p\,d\mu(z)\right\}^{sp^\ast/\omega}
\sum_{k=-\infty}^\infty\delta^{-k(1-ps/\omega)p^\ast}\mu(E_k\setminus E_{k-1})\\
&\quad+r_0^{sp^\ast}\delta^{-k_0p^\ast}\mu(B_0)\\
&\lesssim b^{-sp^\ast/\omega}
\left\{\int_{\delta^{-1}B_0}[g(z)]^p\,d\mu(z)\right\}^{sp^\ast/\omega}
\left\{\int_{\delta^{-1}B_0}[g(z)]^p\,d\mu(z)\right\}\\
&\quad +\mu(B_0)r_0^{sp^\ast} (br_0^\omega)^{-p^\ast/p}
\left\{\int_{\delta^{-1}B_0}[g(z)]^p\,d\mu(z)\right\}^{p^\ast/p}\\
&\lesssim \left[1+\frac{\mu(B_0)}{br_0^\omega}\right]b^{-sp^\ast/\omega}
\left\{\int_{\delta^{-1}B_0}[g(z)]^p\,d\mu(z)\right\}^{p^\ast/p}\\
&\lesssim \frac{\mu(B_0)}{br_0^\omega}b^{-sp^\ast/\omega}
\left\{\int_{\delta^{-1}B_0}[g(z)]^p\,d\mu(z)\right\}^{p^\ast/p},
\end{align*}
which implies that
$$\left[\oint_{B_0}|u(y)-\widetilde{c}|^{p^\ast}\,d\mu(y)\right]^{\frac{1}{p^\ast}}
\lesssim (br_0^\omega)^{-1/p^\ast}b^{-s/\omega}
\left\{\int_{\delta^{-1}B_0}[g(y)]^p\,d\mu(y)\right\}^{1/p}.$$
Recalling that $b=(4A_0)^{-\omega}\delta^\omega
r_0^{-\omega}\mu(\delta^{-1}B_0)$, then we conclude that
$$(br_0^\omega)^{-1/p^\ast}b^{-s/\omega}
\lesssim r_0^s[\mu(\delta^{-1}B_0)]^{-1/p}.$$
This finishes the proof of Lemma \ref{pc}.
\end{proof}

\begin{remark}\label{replace}
Let $p^\ast$, $u$, and $B_0$ be as in Lemma \ref{pc}.
If $p^\ast\in[1,\infty)$, then $u\in L^1(B_0)$ and, moreover, the left hand side
of \eqref{pc-eq} can be replaced by
$$\left[\oint_{B_0}|u(y)-u_{B_0}|^{p^\ast}\,d\mu(y)\right]^{\frac{1}{p^\ast}},$$
where
$$u_{B_0}:=\oint_{B_0}u(y)\,d\mu(y).$$
Indeed, we can find a $c_0\in\rr$ such that
$$\left[\oint_{B_0}|u(y)-c_0|^{p^\ast}\,d\mu(y)\right]^{\frac{1}{p^\ast}}
\leq2\inf_{c\in\rr}\left[\oint_{B_0}|u(y)-c|^{p^\ast}\,d\mu(y)\right]^{\frac{1}{p^\ast}}.$$
By the H\"older inequality, we have
\begin{align*}
\left[\oint_{B_0}|u(y)-u_{B_0}|^{p^\ast}\,d\mu(y)\right]^{\frac{1}{p^\ast}}
&=\left[\oint_{B_0}
\left|u(y)-\oint_{B_0}u(z)\,d\mu(z)\right|^{p^\ast}\,d\mu(y)\right]^{\frac{1}{p^\ast}}\\
&\leq \left[\oint_{B_0}\oint_{B_0}
|u(y)-u(z)|^{p^\ast}\,d\mu(z)\,d\mu(y)\right]^{\frac{1}{p^\ast}}\\
&\lesssim \left[\oint_{B_0}|u(y)-c_0|^{p^\ast}\,d\mu(y)\right]^{\frac{1}{p^\ast}}.
\end{align*}
This finishes the proof of the above claim.
\end{remark}

The next result is a consequence of Lemma \ref{pc} and highlights
the fact that functions in $\dot{M}^s_{p,q}(\cx)$ are actually locally
integrable whenever $p\in(\omega/(\omega+s),\infty)$.

\begin{corollary}\label{pc-int}
Let $s\in(0,\infty)$, $q\in(0,\infty]$, and $p\in(\omega/(\omega+s),\infty)$, where $\omega$ as in \eqref{eq-doub}.
Then every function in $\dot{M}^s_{p,q}(\cx)$ is locally integrable on $\cx$.
\end{corollary}

\begin{proof}
Fix $u\in\dot{M}^s_{p,q}(\cx)$ and observe that, if $\{g_k\}_{k\in\zz}\in\dd^s(u)$, then $g\in\cd^s(u)$, where
$$g:=\left(\sum\limits_{k=-\infty}^\infty g_k^q\right)^{1/q}$$
with the usual modification when $q=\infty$. Thus, $u\in\dot{M}^{s,p}(\cx)$.
Consider any ball  $B_0\subset\cx$ and suppose that $A_0\delta^{p/\omega}<1$.
If we choose $t\in(\omega/(\omega+s),p\wedge(\omega/s))$, then $u\in\dot{M}^{s,t}(\delta^{-1}B_0)$
and Lemma~\ref{pc} implies $u\in L^{t^\ast}(B_0)$, where $t^\ast=\frac{\omega t}{\omega-st}>1$.
Thus, $u\in L^1(B_0)$, which completes the proof of Corollary \ref{pc-int}.
\end{proof}
	
The following result also follows from Lemma \ref{pc}.

\begin{corollary}\label{l-past}
Let $s\in(0,\fz)$, $p\in(\omega/(\omega+s),\omega/s)$,
and  $p^\ast:=\frac{\omega p}{\omega-sp}$ with $\omega$ as in \eqref{eq-doub}.
Assume that $\cx$ has a weak lower bound $Q=\omega$.
Then there exists a constant $C\in\rr$
such that, for any $u\in \dot{M}^{s,p}(\cx)$, $u-C\in L^{p^\ast}(\cx)$ and
\begin{equation}\label{s-p}
\|u-C\|_{L^{p\ast}(\cx)} \leq \widetilde{C} \|u\|_{\dot{M}^{s,p}(\cx)},
\end{equation}
where $\widetilde{C}$  is a positive constant independent of $u$.
\end{corollary}

\begin{proof}
Let $s$ and $p$ be as in this corollary. Let $u\in \dot{M}^{s,p}(\cx)$
and fix a point $x_0\in\cx$.
For any $k\in\nn$, let $B_k:=B(x_0,k)$. Choose a $g\in\mathcal{D}^s(u)$
such that $\|g\|_{L^p(\cx)}\leq 2\|u\|_{\dot{M}^{s,p}(\cx)}$.
By Lemma \ref{pc}, we find that, for any $k\in\nn$,
\begin{align*}
\left[\oint_{B_k}|u(y)-u_{B_k}|^{p^\ast}\,d\mu(y)\right]^{\frac{1}{p^\ast}}
\lesssim k^s\left\{\oint_{\delta^{-1}B_k}[g(y)]^p\,d\mu(y)\right\}^{1/p}
\lesssim k^s[\mu(B_k)]^{-1/p}\|u\|_{\dot{M}^{s,p}(\cx)}.
\end{align*}
This, together with the assumption that $\cx$
has a weak lower bound $Q=\omega$, and Proposition \ref{p2.15}, implies that
\begin{align}\label{3.13x}
\left[\int_{B_k}|u(y)-u_{B_k}|^{p^\ast}\,d\mu(y)\right]^{\frac{1}{p^\ast}}
&\lesssim k^s[\mu(B_k)]^{1/p^\ast-1/p}\|u\|_{\dot{M}^{s,p}(\cx)}\notag\\
&\lesssim k^{s(1-Q/\omega)}\|u\|_{\dot{M}^{s,p}(\cx)}\lesssim\|u\|_{\dot{M}^{s,p}(\cx)}.
\end{align}
By this, we find that, for any $k\in\nn$,
\begin{align*}
\left|u_{B_k}-u_{B_1}\right|
&\leq\frac{1}{\mu(B_1)}\int_{B_1}\left|u(y)-u_{B_k}\right|\,d\mu(y)\\
&\leq \frac{1}{[\mu(B_1)]^{1/p^\ast}}\left[\int_{B_1}
\left|u(y)-u_{B_k}\right|^{p^\ast}\,d\mu(y)\right]^{\frac{1}{p^\ast}}
\leq\frac{1}{[\mu(B_1)]^{1/p^\ast}}\|u\|_{\dot{M}^{s,p}(\cx)},
\end{align*}
which implies that $\{u_{B_k}\}_{k\in\nn}\subset \rr$ is a bounded sequence.
From this,
we deduce that there exist a subsequence
$\{u_{B_{k_j}}\}_{j\in\nn}$ and a constant $C\in\rr$
such that $C=\lim_{j\to\infty}u_{B_{k_j}}$.
Moreover, by \eqref{3.13x} and the Fatou lemma, we further conclude that
\begin{align*}\left[\int_{\cx}|u(x)-C|^{p^\ast}\,d\mu(x)\right]^{\frac{1}{p^\ast}}
&=\left[\int_{\cx}\lim_{j\to\infty}\left|[u(x)-u_{B_k}]
\mathbf{1}_{B_{k_j}}(x)\right|^{p^\ast}\,d\mu(x)\right]^{\frac{1}{p^\ast}}\\
&\leq\varliminf_{j\to\infty}\left[\int_{B_{k_j}}\left|u(x)-u_{B_{k_j}}\right|^{p^\ast}\,d\mu(x)\right]^{\frac{1}{p^\ast}}
\lesssim\|u\|_{\dot{M}^{s,p}(\cx)}.
\end{align*}
This finishes the proof of Corollary \ref{l-past}.
\end{proof}

\begin{remark}
Let $\omega$ be as in \eqref{eq-doub} and $p\in(0,\omega)$.
In \cite[Theorem 22]{agh20}, Alvarado et al. proved that,
if $\cx$ is uniformly perfect (see \cite[(39)]{agh20}),
then that \eqref{s-p} holds true with $s=1$ is equivalent
to that $\cx$ has a lower bound.
\end{remark}

The following lemma is a  Poincar\'e type inequality
for $\dd^s(u)$ (see also \cite[Lemma 2.1]{kyz11}).

\begin{lemma}\label{vpc}
Let $s\in(0,\fz)$. Then there exists a positive constant $C$
such that, for any $k\in\zz$, any measurable function $u$ on
$\cx$, $x\in\cx$, and $\{g_j\}_{j\in\zz}\in\dd^s(u)$,
\begin{equation}\label{eq-vpc}
\inf_{c\in\rr}\oint_{B(x,\delta^k)}|u(y)-c|\,d\mu(y)\leq C\delta^{ks}
\sum_{j=k-3}^{k-1}\oint_{B(x,\delta^{k-2})}g_j(y)\,d\mu(y).
\end{equation}
\end{lemma}

\begin{proof}
Let $s$, $u$, and $\{g_j\}_{j\in\zz}$ be as in this lemma.
Observe that, for any $k\in\zz$ and $x\in\cx$,
\begin{align}\label{3.19x}
\inf_{c\in\rr}\oint_{B(x,\delta^k)}|u(y)-c|\,d\mu(y)
&\leq\oint_{B(x,\delta^k)}|u(y)-u_{B(x,\delta^{k-2})\setminus B(x,A_0\delta^{k-1})}|\,d\mu(y)\notag\\
&\leq\oint_{B(x,\delta^k)}\oint_{B(x,\delta^{k-2})
\setminus B(x,A_0\delta^{k-1})}|u(y)-u(z)|\,d\mu(y)\mu(z).
\end{align}
Note that, due to the fact that $\delta$ is very small, we have,
for any $y \in B(x,\delta^k)$ and $z\in B(x,\delta^{k-2})\setminus B(x,A_0\delta^{k-1})$,
$$d(y,z)\leq A_0[d(y,x)+d(x,z)]< 2A_0\delta^{k-2}\leq\delta^{k-3}$$
and
$$d(x,z)\leq A_0[d(x,y)+d(y,z)]<A_0\delta^k+A_0d(y,z),$$
which implies that
$$\delta^k\leq d(y,z)<\delta^{k-3}.$$
From this, we deduce that there exists a unique $j_0\in\{k-1,k-2,k-3\}$ such that
$$\delta^{j_0+1}\leq d(y,z)<\delta^{j_0}$$
and hence
$$
|u(y)-u(z)|\leq [d(y,z)]^s[g_{j_0}(y)+g_{j_0}(z)]
\leq \delta^{(k-3)s}\sum_{j=k-3}^{k-1}[g_j(y)+g_j(z)].
$$
Therefore, by this and \eqref{3.19x}, we conclude that
\begin{align*}
&\inf_{c\in\rr}\oint_{B(x,\delta^k)}|u(y)-c|\,d\mu(y)\\
&\quad\lesssim \delta^{ks}\sum_{j=k-3}^{k-1}\oint_{B(x,\delta^k)}
\oint_{B(x,\delta^{k-2})\setminus B(x,A_0\delta^{k-1})}[g_j(y)+g_j(z)]\,d\mu(y)\mu(z)\\
&\quad\lesssim \delta^{ks}\sum_{j=k-3}^{k-1}\oint_{B(x,\delta^{k-2})}g_j(y)\,d\mu(y),
\end{align*}
which completes the proof of Lemma \ref{vpc}.
\end{proof}

\begin{remark}\label{vpc-R}
Similarly to Remark \ref{replace}, under the assumptions same as in
Lemma \ref{vpc}, the left hand side of \eqref{eq-vpc} can be replaced by
$$\oint_{B(x,\delta^k)}|u(y)-u_{B(x,\delta^k)}|\,d\mu(y).$$
\end{remark}

Using Lemma \ref{pc}, we can show the following Poincar\'e type inequality,
which is very useful in the case when $p\in(0,1]$.

\begin{lemma}\label{vpc2}
Let $s\in(0,\fz)$, $p\in(0,1]$, and $\varepsilon,\ \varepsilon'\in(0,s)$
with $\varepsilon<\varepsilon'$.
If $A_0\delta^{p/\omega}<1$, then there exists a positive
constant $C$ such that, for any $k\in\zz$,
$x\in\cx$, any measurable function $u$, and
$\{g_j\}_{j\in\zz}\in\dd^s(u)$,
\begin{align}\label{eq-vpc2}
&{}\inf_{c\in\rr}\left[\oint_{B(x,\delta^k)}|u(y)-c|^{\frac{\omega p}{\omega-\varepsilon p}}
\,d\mu(y)\right]^{\frac{\omega-\varepsilon p}{\omega p}}\notag\\
&\quad\leq C\delta^{k\varepsilon'}\sum_{j=k-2}^\infty\delta^{j(s-\varepsilon')}
\left\{\oint_{B(x,\delta^{k-1})}[g_j(y)]^p\,d\mu(y)\right\}^{\frac{1}{p}}.
\end{align}
\end{lemma}

\begin{proof} Let all the notation be the same as in this lemma.
Without loss of generality, we may assume that the
right hand side of \eqref{eq-vpc2} is less than infinity.
For any $k\in\zz$ and $x\in\cx$, let
$$g(x):=\left\{\sum_{j=k-2}^\infty\delta^{j(s-\varepsilon)p}[g_j(x)]^p\right\}^{\frac{1}{p}}.$$
We claim that $g\in\cd^s(u)$ and $u\in\dot{M}^{\varepsilon,p}(B(x,\delta^{k-1}))$.
Indeed, for any $y,\ z\in B(x,\delta^{k-1})$, we have
$$d(y,z)\leq A_0[d(y,x)+d(x,z)]<2A_0\delta^{k-1}<\delta^{k-2}.$$
Therefore, there exists a unique integer $j_0 \geq k-2$ such that
$$\delta^{j_0+1}\leq d(y,z)<\delta^{j_0}.$$
Then, since $\{g_j\}_{j\in\zz}\in\dd^s(u)$,
it follows that there exists an $E\subset\cx$ with $\mu(E)=0$
such that, for any $y,\ z\in B(x,\delta^{k-1})\setminus E$,
\begin{align*}
|u(y)-u(z)|&\leq[d(y,z)]^s[g_{j_0}(y)+g_{j_0}(z)]\leq
[d(y,z)]^\varepsilon\delta^{j(s-\varepsilon)}[g_{j_0}(y)+g_{j_0}(z)]\\
&\leq [d(y,z)]^\varepsilon[g(y)+g(z)],
\end{align*}
which implies that $g\in\cd^\varepsilon(u)$.
On the other hand, by $p\in(0,1]$, the H\"older inequality with exponent $1/p\geq1$,
and $0<\varepsilon<\varepsilon'<s$, we know that
\begin{align*}
\|g\|_{L^p(B(x,\delta^{k-1}))}
&=\left\{\int_{B(x,\delta^{k-1})}
\sum_{j=k-2}^\infty\delta^{j(s-\varepsilon)p}[g_j(y)]^p\,d\mu(y)\right\}^{\frac{1}{p}}\\
&=\left\{\sum_{j=k-2}^\infty
\delta^{j(\varepsilon'-\varepsilon)p}\delta^{j(s-\varepsilon')p}
\int_{B(x,\delta^{k-1})}[g_j(y)]^p\,d\mu(y)\right\}^{\frac{1}{p}}\\
&\leq \left[\sum_{j=k-2}^\infty
\delta^{j(\varepsilon'-\varepsilon)p(1/p)'}\right]^{\frac{1}{(1/p)'p}}
\sum_{j=k-2}^\infty\delta^{j(s-\varepsilon')}
\left\{\int_{B(x,\delta^{k-1})}[g_j(y)]^p\,d\mu(y)\right\}^{\frac{1}{p}}\\
&\lesssim \delta^{k(\varepsilon'-\varepsilon)}
\left[V_{\delta^{k-1}}(x)\right]^{1/p}\sum_{j=k-2}^\infty
\delta^{j(s-\varepsilon')}
\left\{\oint_{B(x,\delta^{k-1})}[g_j(y)]^p\,d\mu(y)\right\}^{\frac{1}{p}}<\infty.
\end{align*}
From this, we deduce that $u\in\dot{M}^{\varepsilon,p}(B(x,\delta^{k-1}))$.
Combining the above claim and Lemma \ref{pc}, and
using the H\"older inequality with exponent $1/p\geq1$, we find that
\begin{align*}
&\inf_{c\in\rr}\left[\oint_{B(x,\delta^k)}|u(y)-c|^{\frac{\omega p}{\omega-\varepsilon p}}
\,d\mu(y)\right]^{\frac{\omega-\varepsilon p}{\omega p}}\\
&\quad\lesssim \delta^{k\varepsilon}
\left\{\oint_{B(x,\delta^{k-1})}[g(y)]^p\,d\mu(y)\right\}^{1/p}\\
&\quad\sim \delta^{k\varepsilon}
\left\{\sum_{j=k-2}^\infty\delta^{j(s-\varepsilon)p}
\oint_{B(x,\delta^{k-1})}[g_j(y)]^p\,d\mu(y)\right\}^{1/p}\\
&\quad\lesssim \delta^{k\varepsilon}\left\{\sum_{j=k-2}^\infty
\delta^{j(\varepsilon'-\varepsilon)p(1/p)'}\right\}^{\frac{1}{(1/p)'p}}
\sum_{j=k-2}^\infty\delta^{j(s-\varepsilon')}
\left\{\oint_{B(x,\delta^{k-1})}[g_j(y)]^p\,d\mu(y)\right\}^{\frac{1}{p}}\\
&\quad\lesssim\delta^{k\varepsilon'}\sum_{j=k-2}^\infty\delta^{j(s-\varepsilon')}
\left\{\oint_{B(x\delta^{k-1})}[g_j(y)]^p\,d\mu(y)\right\}^{\frac{1}{p}}.
\end{align*}
This finishes the proof of Lemma \ref{vpc2}.
\end{proof}

The following lemma illustrates  that  any  element of $\ca\dot{F}^s_{p,\infty}(\cx)$
is a locally integrable function.

\begin{lemma}\label{btl-loc}
Let $\beta,\ \gamma\in (0,\eta)$ with $\eta$ as in Definition \ref{exp-ati},
$s\in (0, \beta\wedge\gamma)$, and $p\in(\omega/(\omega+s),\infty)$
with $\omega$ as in \eqref{eq-doub}.
Assume that the measure $\mu$ of $\cx$ has a weak lower bound $Q=\omega$.
Then, for any $f\in\ca\dot{F}^s_{p,\infty}(\cx)$,
there exists an $\widetilde{f}\in L^1_{\loc}(\cx)$ such that
$f=\widetilde{f}$ in $(\cgg)'.$
\end{lemma}

To prove Lemma \ref{btl-loc}, we need the notion of
approximations of the identity with exponential
decay and integration 1; see \cite[Definition 2.8]{hhllyy} for more details.

\begin{definition}\label{1-exp-ati}
Let $\eta\in(0,1)$ be as in Definition \ref{exp-ati}.
A sequence $\{Q_k\}_{k\in\zz}$ of bounded linear integral operators
on $L^2(\cx)$ is called  an \emph{approximation of the identity with exponential
decay and integration 1} (for short, 1-exp-ATI) if $\{Q_k\}_{k\in\zz}$
has the following properties:
\begin{enumerate}
\item[{\rm(i)}] for any $k\in\zz$, $Q_k$
satisfies (ii), (iii), and (iv) of Definition \ref{exp-ati} but without the decay factor
$$\exp\left\{-\nu\left[\frac{\max\{d(x,\cy^k),d(y,\cy^k)\}}{\delta^k}\right]^a\right\};$$
\item[{\rm(ii)}] for any $k\in\zz$ and $x\in\cx$,
$$\int_\cx Q_k(x,y)\,d\mu(y)=1=\int_\cx Q_k(y,x)\,d\mu(y);$$
\item[{\rm(iii)}] letting, for any $k\in\zz$,
$P_k:=Q_k-Q_{k-1}$, then $\{P_k\}_{k\in\zz}$ is an exp-ATI.
\end{enumerate}
\end{definition}

\begin{remark} As was pointed out in
\cite[Remark 2.9]{hhllyy}, the existence of the 1-exp-ATI is guaranteed by \cite[Lemma 10.1]{ah13}.
Moreover, by the proofs of \cite[Proposition 2.9]{hlyy} and
\cite[Proposition 2.7(iv)]{hmy08}, we know that, if $\{Q_k\}_{k\in\zz}$ is a 1-exp-ATI, then, for any $f\in L^2(\cx)$,
$\lim_{k\to\infty}Q_kf=f$ in $L^2(\cx)$.
\end{remark}

By an argument similar to that used in the proof of \cite[Proposition 2.10]{hlyy},
we have the following conclusion; we omit the details.

\begin{lemma}\label{expfk}
Let $\beta,\ \gamma\in(0,\eta)$ with $\eta$ as in Definition \ref{exp-ati}, $s\in(0,\beta\wedge\gamma)$, $k\in\zz$, $x$, $y\in\cx$, and $\{Q_k\}_{k\in\zz}$ be a 1-exp-ATI. For any $z\in\cx$, let
$$\phi(z):=\delta^{ks}[d(x,y)]^{-s}[Q_k(x,z)-Q_k(y,z)].$$
If  $d(x,y)\in (0,\delta^k]$, then $\phi\in\cf_k(x)$, where $\cf_k(x)$ is as in Definition \ref{gbtl}.
\end{lemma}

Next, we prove Lemma \ref{btl-loc}.

\begin{proof}[Proof of Lemma \ref{btl-loc}]
Assume that $f\in\ca\dot{F}^s_{p,\infty}(\cx)$ and $\{Q_k\}_{k\in\zz}$ is an 1-exp-ATI.
For any $x\in\cx$, let
$$g(x):=\sup_{k\in\zz}\delta^{-ks}\sup_{\phi\in\cf_k(x)}|\langle f,\phi\rangle|.$$
First, for any $k\in\zz$, $i\in\nn$, and $x\in\cx$, we have
\begin{align*}
|Q_{k}f(x)-Q_{k+i}f(x)|\leq\sum_{j=0}^{i-1}|Q_{k+j}f(x)-Q_{k+j+1}f(x)|
=\sum_{j=0}^{i-1}|\langle f,Q_{k+j}(x,\cdot)-Q_{k+j+1}(x,\cdot)\rangle|.
\end{align*}
Note that, by \eqref{10.23.3} and \eqref{10.23.4},
we obtain, for any $k\in\zz$, and $x\in\cx$,
$$Q_{k+j}(x,\cdot)-Q_{k+j+1}(x,\cdot)\in\cf_{k+j+1}(x).$$
From this, we deduce that
\begin{align}\label{q-ca}
|Q_{k}f(x)-Q_{k+i}f(x)|\leq \sum_{j=0}^{i-1}\delta^{(k+j+1)s}g(x)\lesssim \delta^{ks}g(x).
\end{align}
We now consider the case $p\in(1,\infty)$ first.
To this end, note that \eqref{q-ca} implies
that $\{Q_{k}f-Q_{k+i}f\}_{i\in\nn}$ is
a Cauchy sequence in $L^p(\cx)$.
By the completeness of $L^p(\cx)$,
we find that there exists an $f_k\in L^p(\cx)$ such that
$$\lim_{i\to\infty}(Q_{k}f-Q_{k+i}f)=f_k$$
both in $L^p(\cx)$ and pointwise.
Observe that, for any $k,\ k'\in\zz$, we have
\begin{align*}
f_k&=\lim_{i\to\infty}(Q_{k}f-Q_{k+i}f)
=Q_{k}f-Q_{k'}f+\lim_{i\to\infty}(Q_{k'}f-Q_{k+i}f)\\
&=Q_{k}f-Q_{k'}f+f_{k'}
\end{align*}
both in $L^p(\cx)$ and pointwise.
Let $\widetilde{f}:=Q_0f-f_0$. Then, by this,  we conclude that
$\widetilde{f}\in L^1_{\loc}(\cx)$ and,
for any $k\in\zz$, $\widetilde{f}=Q_kf-f_k$.
On the other hand, since $Q_kf\to f$ in $(\cgg)'$ as $k\to\infty$,
it follows that, for any $\psi\in\cgg$,
\begin{align*}
\langle \widetilde{f},\psi\rangle&=\int_{\cx}\widetilde{f}(x)\psi(x)\,d\mu(x)\\
&=\int_{\cx}
\left\{Q_0f(x)-\lim_{i\to \infty}[Q_0f(x)-Q_if(x)]\right\}\psi(x)\,d\mu(x)\\
&=\int_{\cx}
Q_0f(x)\psi(x)\,d\mu(x)-
\lim_{i\to \infty}\int_{\cx}[Q_0f(x)-Q_if(x)]\psi(x)\,d\mu(x)\\
&=\lim_{i\to \infty}\int_{\cx}
Q_if(x)\psi(x)\,d\mu(x)
=\lim_{i\to \infty}\left\langle Q_if, \psi\right\rangle = \langle f, \psi\rangle,
\end{align*}
which completes the proof of this lemma in the case when $p\in(1,\infty)$.

Next, we consider the case $p\in(\omega/(\omega+s),1]$.
For any $x,\ y\in\cx$, let $k_0\in\zz$ be such that $\delta^{k_0+1}<d(x,y)\leq \delta^{k_0}$.
Then, by Lemma \ref{expfk} and \eqref{q-ca}, we know that, for any $k\in\zz$ with $k>k_0$,
\begin{align}\label{q-gi}
|Q_kf(x)-Q_kf(y)|&\leq|Q_kf(x)-Q_{k_0}f(x)|+
|Q_{k_0}f(x)-Q_{k_0}f(y)|\noz\\
&\quad+|Q_{k_0}f(y)-Q_kf(y)|\noz\\
&\lesssim \delta^{k_0s}[g(x)+g(y)]\lesssim [d(x,y)]^s[g(x)+g(y)].
\end{align}
On the other hand, we find that, for any $k\in\zz$ with $k\leq k_0$,
$$|Q_kf(x)-Q_kf(y)|=[d(x,y)]^s\delta^{-ks}\left\langle f, \delta^{ks}
[d(x,y)]^{-s}[Q_k(x,\cdot)-Q_k(y,\cdot)]\right\rangle.$$
Note that, due to $k\leq k_0$, then $d(x,y)\leq  \delta^k$.
From this and Lemma \ref{expfk}, we deduce that
$$\delta^{ks}[d(x,y)]^{-s}[Q_k(x,\cdot)-Q_k(y,\cdot)]\in\cf_k(x),$$
which implies that
\begin{equation}\label{q-gii}
|Q_kf(x)-Q_kf(y)|\lesssim [d(x,y)]^s[g(x)+g(y)].
\end{equation}
Combining \eqref{q-gi} and \eqref{q-gii}, we know that, for any $k\in\zz$,
\begin{equation}\label{q-g}
|Q_kf(x)-Q_kf(y)|\lesssim [d(x,y)]^s[g(x)+g(y)].
\end{equation}
Moreover, by $f\in\ca\dot{F}^s_{p,\infty}(\cx)$, we know that $g\in L^p(\cx)$.
From this and \eqref{q-g}, we further deduce that, for any $k\in\zz$,
$$Q_kf\in \dot{M}^{s,p}(\cx)\quad\text{and}\quad\|Q_kf\|_{\dot{M}^{s,p}(\cx)}
\lesssim \|g\|_{L^p(\cx)}\sim\|f\|_{\ca\dot{F}^s_{p,\infty}(\cx)}.$$
By this and Corollary \ref{l-past}, we find that, for any $k\in\zz$,
there exists a constant $C_k$ such that
$$Q_kf-C_k\in L^{p^\ast}(\cx)\quad\text{and}\quad\|Q_kf-C_k\|_{L^{p^\ast}(\cx)}
\lesssim \|Q_kf\|_{\dot{M}^{s,p}(\cx)}\lesssim\|f\|_{\ca\dot{F}^s_{p,\infty}(\cx)}.$$
From this and the weak compactness property of $L^{p^\ast}(\cx)$ (recall that $p^\ast>1$ in this case),
we deduce that there exist a subsequence $\{Q_{k_j}f-C_{k_j}\}_{j\in\nn}$
and a function $\widetilde{f}\in L^{p^\ast}(\cx)$ such that
$$\widetilde{f}=\lim_{j\to\infty}[Q_{k_j}f-C_{k_j}]$$
both weakly in $L^{p^\ast}(\cx)$ and also in $(\cgg)'$.
Besides, by \eqref{q-ca},
we find that, for any $j,\ j'\in\nn$, $Q_{k_j}f-Q_{k_{j'}}f\in L^p(\cx)$
and $[Q_{k_j}f-C_{k_{j}}]-[Q_{k_{j'}}f-C_{k_{j'}}]\in L^{p^\ast}(\cx)$,
which further implies that $C_{k_{j}}=C_{k_{j'}}$.
By the fact that $Q_{k_j}f\to f$ in $(\cgg)'$ as $j\to\infty$,
we conclude that
$$f=f-C_{k_0}=\lim_{j\to\infty}[Q_{k_j}f-C_{k_j}]=\widetilde{f}$$
in $(\cgg)'$. This finishes the proof of Lemma \ref{btl-loc}.
\end{proof}

We now present the proof of Theorem \ref{hbtl-gbtl}.
\begin{proof}[Proof of Theorem \ref{hbtl-gbtl}]
By similarity, we only prove (i).
We first show $\dot{M}^s_{p,q}(\cx)\subset \ca\hf$.
To this end, fix $u\in\dot{M}^s_{p,q}(\cx)$ and recall that $u$
is locally integrable on $\cx$ by Corollary~\ref{pc-int}. With this in mind, we
consider five cases on $p$ and $q$.

{\it Case 1)} $p\in (1,\infty)$ and $q\in (1,\infty]$.
In this case, we only consider the case $q\in(1,\infty)$,
because the proof of the case $q=\infty$ is similar to that of $q\in(1,\infty)$.
Choose $\{g_k\}_{k\in\zz}\in \dd^s(u)$
such that
$$\left\|\left(\sum_{k=-\infty}^\infty g_k^q\right)^{\frac{1}{q}}\right\|_{L^p(\cx)}
\lesssim \|u\|_{\dot{M}^s_{p,q}(\cx)}.$$
Then we have, for any $k\in\zz$, $x\in\cx$, and $\phi\in\cf_k(x)$,
\begin{align*}
|\langle u,\phi\rangle|&=\left|\int_{\cx}\phi(y)u(y)\,d\mu(y)\right|
=\left|\int_{\cx}\phi(y)\left[u(y)-u_{B(x,\delta^{k})}\right]\,d\mu(y)\right|\\
&\leq \int_{\cx}D_\gamma\left(x,y;\delta^k\right)
\left|u(y)-u_{B(x,\delta^{k})}\right|\,d\mu(y)\\
&\lesssim \sum_{j=0}^\infty\delta^{j\gamma}\oint_{B(x,\delta^{k-j})}
\left|u(y)-u_{B(x,\delta^{k})}\right|\,d\mu(y),
\end{align*}
where $D_\gamma(x,y;\delta^k)$ is as in \eqref{decay}.
Notice that, for any $k\in\zz$, $j\in\zz_+$, and $x\in\cx$,
\begin{align*}
&\oint_{B(x,\delta^{k-j})}\left|u(y)-u_{B(x,\delta^{k})}\right|\,d\mu(y)\\
&\quad= \oint_{B(x,\delta^{k-j})}\left|u(y)-u_{B(x,\delta^{k-1})}
+u_{B(x,\delta^{k-1})}-u_{B(x,\delta^{k})}\right|\,d\mu(y)\\
&\quad\leq \oint_{B(x,\delta^{k-j})}\left|u(y)-u_{B(x,\delta^{k-1})}\right|\,d\mu(y)
+\frac{1}{V_{\delta^k}(x)}\int_{B(x,\delta^{k})}\left|u(y)-u_{B(x,\delta^{k-1})}\right|\,d\mu(y)\\
&\quad\lesssim \oint_{B(x,\delta^{k-j})}\left|u(y)-u_{B(x,\delta^{k-1})}\right|\,d\mu(y)
+\oint_{B(x,\delta^{k-1})}\left|u(y)-u_{B(x,\delta^{k-1})}\right|\,d\mu(y)\\
&\quad\lesssim \sum_{i=0}^{j} \oint_{B(x,\delta^{k-i})}
\left|u(y)-u_{B(x,\delta^{k-i})}\right|\,d\mu(y).
\end{align*}
Combining the above two inequalities, we obtain, for any $k\in\zz$,
\begin{align}\label{u-phi}
|\langle u,\phi\rangle|&\lesssim \sum_{j=0}^\infty\delta^{j\gamma}
\sum_{i=0}^{j}\oint_{B(x,\delta^{k-i})}\left|u(y)-u_{B(x,\delta^{k-i})}\right|\,d\mu(y)\notag\\
&\sim \sum_{i=0}^\infty
\sum_{j=i}^{\infty}\delta^{j\gamma}\oint_{B(x,\delta^{k-i})}
\left|u(y)-u_{B(x,\delta^{k-i})}\right|\,d\mu(y)\notag\\
&\lesssim \sum_{i=0}^\infty\delta^{i\gamma}\oint_{B(x,\delta^{k-i})}
\left|u(y)-u_{B(x,\delta^{k-i})}\right|\,d\mu(y).
\end{align}
From this and Lemma \ref{vpc} (see also Remark~\ref{vpc-R}), we deduce that, for any $k\in\zz$,
\begin{align}\label{u-phi2}
|\langle u,\phi\rangle|&\lesssim \sum_{i=0}^\infty\delta^{i\gamma}
\delta^{(k-i)s}\sum_{j=k-i-3}^{k-i-1}\oint_{B(x,\delta^{k-i-2})}g_j(y)\,d\mu(y)\notag\\
&\lesssim \delta^{ks}\sum_{j=-\infty}^k
\sum_{i=k-j-3}^{k-j-1}\delta^{i(\gamma-s)}
\oint_{B(x,\delta^{k-i-2})}g_j(y)\,d\mu(y)\notag\\
&\lesssim \delta^{k\gamma}\sum_{j=-\infty}^k\delta^{-j(\gamma-s)}
\oint_{B(x,\delta^{j-2})}g_j(y)\,d\mu(y)
\lesssim \delta^{k\gamma}\sum_{j=-\infty}^k\delta^{-j(\gamma-s)}
M(g_j)(x).
\end{align}
By this, the H\"older inequality, and Lemma \ref{fsvv}, we conclude that
\begin{align*}
\|u\|_{\ca\hf}
&\lesssim \left\|\left\{\sum_{k=-\infty}^\infty\delta^{-ksq}
\left[\delta^{k\gamma}\sum_{j=-\infty}^k\delta^{-j(\gamma-s)}
M(g_j)\right]^q\right\}^{1/q}\right\|_{L^p(\cx)}\\
&\lesssim \left\|\left\{\sum_{k=-\infty}^\infty\delta^{k(\gamma-s)}
\sum_{j=-\infty}^k\delta^{-j(\gamma-s)}[M(g_j)]^q\right\}^{1/q}\right\|_{L^p(\cx)}\\
&\lesssim \left\|\left\{\sum_{j=-\infty}^\infty[M(g_j)]^q\right\}^{1/q}\right\|_{L^p(\cx)}
\lesssim \left\|\left\{\sum_{j=-\infty}^\infty g_j^q\right\}^{1/q}\right\|_{L^p(\cx)}
\lesssim \|u\|_{\dot{M}^s_{p,q}(\cx)},
\end{align*}
which is the desired estimate in this case.

{\it Case 2)} $p\in(1,\infty)$ and $q\in (\omega/(\omega+s),1]$.
In this case, choose $\{g_k\}_{k\in\zz}\in \dd^s(u)$ as in Case~1).
Recall that we may assume $A_0\delta^{p/\omega}<1$. As such,
combining \eqref{u-phi} and Lemma \ref{vpc2} [with $p=\omega/(\omega+\varepsilon)<1$], we find that
for any fixed $\varepsilon,\ \varepsilon'\in(0,s)$ with $\varepsilon<\varepsilon'$, there holds true
\begin{align}\label{u-phi3}
|\langle u,\phi\rangle|
&\lesssim \sum_{i=0}^\infty\delta^{i\gamma} \delta^{(k-i)\varepsilon'}
\sum_{j=k-i-2}^\infty\delta^{j(s-\varepsilon')}
\left\{\oint_{B(x,\delta^{k-i-1})}[g_j(y)]^{\frac{\omega}{\omega+\varepsilon}}
\,d\mu(y)\right\}^{\frac{\omega+\varepsilon}{\omega}}\notag\\
&\lesssim \delta^{k\varepsilon'}\sum_{j=-\infty}^{k-2}\delta^{j(s-\varepsilon')}
\left[M\left([g_j(x)]^{\frac{\omega}{\omega
+\varepsilon}}\right)\right]^{\frac{\omega+\varepsilon}{\omega}}
\sum_{i=k-j-2}^\infty\delta^{i(\gamma-\varepsilon')}\notag\\
&\quad+\delta^{k\varepsilon'}\sum_{j=k-1}^\infty\delta^{j(s-\varepsilon')}
\left[M\left([g_j(x)]^{\frac{\omega}{\omega+
\varepsilon}}\right)\right]^{\frac{\omega+\varepsilon}{\omega}}
\sum_{i=0}^\infty\delta^{i(\gamma-\varepsilon')}\notag\\
&\lesssim \delta^{k\gamma}\sum_{j=-\infty}^{k-2}\delta^{j(s-\gamma)}
\left[M\left([g_j(x)]^{\frac{\omega}{\omega+
\varepsilon}}\right)\right]^{\frac{\omega+\varepsilon}{\omega}}
+\delta^{k\varepsilon'}\sum_{j=k-1}^\infty\delta^{j(s-\varepsilon')}
\left[M\left([g_j(x)]^{\frac{\omega}{\omega
+\varepsilon}}\right)\right]^{\frac{\omega+\varepsilon}{\omega}}.
\end{align}
Choosing $\varepsilon\in(0,s)$ such that
$\frac{\omega}{\omega+s}<\frac{\omega}{\omega+\varepsilon}<q$, and using \eqref{u-phi3},
\eqref{r} (with $\theta=q$),
and Lemma \ref{fsvv}, we conclude that
\begin{align*}
\|u\|_{\ca\hf}&\lesssim \left\|\left(\sum_{k=-\infty}^\infty\delta^{-ksq}
\left\{\delta^{k\gamma}\sum_{j=-\infty}^{k-2}\delta^{j(s-\gamma)}
\left[M\left([g_j]^{\frac{\omega}{\omega
+\varepsilon}}\right)\right]^{\frac{\omega
+\varepsilon}{\omega}}\right.\right.\right.\\
&\quad\left.\left.\left.+\delta^{k\varepsilon'}\sum_{j=k-1}^\infty\delta^{j(s-\varepsilon')}
\left[M\left([g_j]^{\frac{\omega}{\omega+\varepsilon}}\right)
\right]^{\frac{\omega+\varepsilon}{\omega}}\right\}^q\right)^{1/q}\right\|_{L^p(\cx)}\\
&\lesssim \left\|\left\{\sum_{k=-\infty}^\infty\delta^{-ksq}
\delta^{k\gamma q}\sum_{j=-\infty}^{k-2}\delta^{j(s-\gamma)q}
\left[M\left([g_j]^{\frac{\omega}{\omega+\varepsilon}}\right)\right]^
{\frac{(\omega+\varepsilon)q}{\omega}}\right\}^{1/q}\right\|_{L^p(\cx)}\\
&\quad+\left\|\left\{\sum_{k=-\infty}^\infty\delta^{-ksq}
\delta^{k\varepsilon' q}\sum_{j=k-1}^\infty\delta^{j(s-\varepsilon')q}
\left[M\left([g_j]^{\frac{\omega}{\omega+\varepsilon}}\right)\right]^
{\frac{(\omega+\varepsilon)q}{\omega}}\right\}^{1/q}\right\|_{L^p(\cx)}\\
&\lesssim \left\|\left\{\sum_{j=-\infty}^\infty
\left[M\left([g_j]^{\frac{\omega}{\omega+\varepsilon}}\right)\right]^
{\frac{(\omega+\varepsilon)q}{\omega}}\right\}^{1/q}\right\|_{L^p(\cx)}
\lesssim \left\|\left\{\sum_{j=-\infty}^\infty
g_j^q\right\}^{1/q}\right\|_{L^p(\cx)}\\
&\lesssim \|u\|_{\dot{M}^s_{p,q}(\cx)}.
\end{align*}
This is the desired estimate in this case.

{\it Case 3)} $p\in (\omega/(\omega+s),1]$ and $q\in (\omega/(\omega+s),\infty]$.
In this case, the proof of Theorem \ref{hbtl-gbtl} is
similar to that of Case 2); the details are omitted.

{\it Case 4)} $p=\infty$ and $q\in (1,\infty]$.
In this case, we only consider the case $q\in(1,\infty)$
because the proof of the case $q=\infty$ is similar to that of $q\in(1,\infty)$.
Choose $\{g_k\}_{k\in\zz}\in \dd^s(u)$
such that
\begin{equation}\label{ch-g}
\sup_{k\in\zz}\sup_{x\in\cx}
\left\{\sum_{j=k}^\infty\oint_{B(x,\delta^k)} [g_j(y)]^q\,d\mu(y)\right\}^{\frac{1}{q}}
\lesssim \|u\|_{\dot{M}^s_{\infty,q}(\cx)}.
\end{equation}
By \eqref{u-phi2} and the H\"older inequality, we know that, for any $l\in\zz$ and $x\in\cx$,
\begin{align*}
&\oint_{B(x,2A_0C_0\delta^l)}\sum_{k=l}^\infty\delta^{-ksq}
\sup_{\phi\in\cf_k(z)}|\langle u,\phi\rangle|^q\,d\mu(z)\\
&\quad\lesssim\oint_{B(x,2A_0C_0\delta^l)}\sum_{k=l}^\infty
\delta^{-ksq}\left[\delta^{k\gamma}\sum_{j=-\infty}^k\delta^{-j(\gamma-s)}
\oint_{B(z,\delta^{j-2})}g_j(y)\,d\mu(y)\right]^q\,d\mu(z)\\
&\quad \lesssim \oint_{B(x,2A_0C_0\delta^l)}\sum_{k=l}^\infty\delta^{k(\gamma-s)}
\sum_{j=-\infty}^k\delta^{-j(\gamma-s)}
\left[\oint_{B(z,\delta^{j-2})}g_j(y)\,d\mu(y)\right]^q\,d\mu(z)\\
&\quad \lesssim \oint_{B(x,2A_0C_0\delta^l)}\sum_{k=l}^\infty\delta^{k(\gamma-s)}
\sum_{j=-\infty}^{l-1}\delta^{-j(\gamma-s)}
\left[\oint_{B(z,\delta^{j-2})}g_j(y)\,d\mu(y)\right]^q\,d\mu(z)\\
&\quad\quad+\oint_{B(x,2A_0C_0\delta^l)}\sum_{k=l}^\infty\delta^{k(\gamma-s)}
\sum_{j=l}^k\delta^{-j(\gamma-s)}
\left[\oint_{B(z,\delta^{j-2})}g_j(y)\,d\mu(y)\right]^q\,d\mu(z)\\
&\quad =: \mathrm{Y}_1+\mathrm{Y}_2.
\end{align*}
To estimate $\mathrm{Y}_1$, by the H\"older inequality with exponent $q>1$,
and \eqref{ch-g}, we find that, for any $j\in\zz$ and $z\in\cx$,
$$\left[\oint_{B(z,\delta^{j-2})}g_j(y)\,d\mu(y)\right]^q
\leq \oint_{B(z,\delta^{j-2})}[g_j(y)]^q\,d\mu(y)\lesssim \|u\|_{\dot{M}^s_{\infty,q}(\cx)}^q.$$
From this, we further deduce that
\begin{align*}
\mathrm{Y}_1\lesssim \|u\|_{\dot{M}^s_{\infty,q}(\cx)}^q
\oint_{B(x,2A_0C_0\delta^l)}\sum_{k=l}^\infty\delta^{k(\gamma-s)}
\sum_{j=-\infty}^{l-1}\delta^{-j(\gamma-s)}\,d\mu(z)
\lesssim \|u\|_{\dot{M}^s_{\infty,q}(\cx)}^q.
\end{align*}
To estimate $\mathrm{Y}_2$, note that, for any $j\in\zz$ with $j\geq l$,
$z\in B(x,2A_0C_0\delta^l)$, and $y\in B(z,\delta^{j-2})$, we have
$$d(y,x)\leq A_0[d(y,z)+d(z,x)]<A_0\delta^{j-2}+2A_0^2C_0\delta^l
\leq(A_0\delta+2A_0^2C_0\delta^3)\delta^{l-3}\leq\delta^{l-3},$$
which further implies that
\begin{align*}
\mathrm{Y}_2&\lesssim \oint_{B(x,2A_0C_0\delta^l)}
\sum_{k=l}^\infty\delta^{k(\gamma-s)}
\sum_{j=l}^k\delta^{-j(\gamma-s)}\left[\oint_{B(z,\delta^{j-2})}
g_j(y)\mathbf{1}_{B(x, {\delta^{l-3}})}(y)\,d\mu(y)\right]^q\,d\mu(z)\\
&\lesssim \oint_{B(x,2A_0C_0\delta^l)}\sum_{k=l}^\infty\delta^{k(\gamma-s)}
\sum_{j=l}^k\delta^{-j(\gamma-s)}
\left[M\left(g_j\mathbf{1}_{B(x,
\delta^{l-3})}\right)(z)\right]^q\,d\mu(z)\\
&\lesssim \sum_{j=l}^\infty\oint_{B(x,2A_0C_0\delta^l)}
\left[M\left(g_j\mathbf{1}_{B(x,
\delta^{l-3})}\right)(z)\right]^q\,d\mu(z)\\
&\lesssim \sum_{j=l}^\infty\oint_{B(x,2A_0C_0\delta^l)}
\left[g_j(z)\mathbf{1}_{B(x,\delta^{l-3})}(z)\right]^q\,d\mu(z)\\
&\lesssim \sum_{j=l-3}^\infty\oint_{B(x,\delta^{l-3})}
\left[g_j(z)\right]^q\,d\mu(z)\lesssim \|u\|_{\dot{M}^s_{\infty,q}(\cx)}^q.
\end{align*}
Combining the estimates of $\mathrm{Y}_1$ and $\mathrm{Y}_2$,
we conclude that, for any $l\in\zz$ and $x\in\cx$,
$$\oint_{B(x,2A_0C_0\delta^l)}\sum_{k=l}^\infty\delta^{-ksq}
\sup_{\phi\in\cf_k(z)}|\langle u,\phi\rangle|^q\,d\mu(z)
\lesssim \|u\|_{\dot{M}^s_{\infty,q}(\cx)}^q,$$
which, together with Lemma~\ref{2-cube}(iii), implies
the desired estimate in this case.

{\it Case 5)} $p=\infty$ and $q\in (\omega/(\omega+s),1]$.
In this case, choose $\{g_k\}_{k\in\zz}\in \dd^s(u)$ as in Case 4).
Arguing as in \eqref{u-phi3}, we have, for any fixed $\varepsilon,\ \varepsilon'\in(0,s)$
with $\varepsilon<\varepsilon'$,
\begin{align*}
&\oint_{B(x,2A_0C_0\delta^l)}\sum_{k=l}^\infty\delta^{-ksq}
\sup_{\phi\in\cf_k(z)}|\langle u,\phi\rangle|^q\,d\mu(z)\\
&\quad \lesssim\oint_{B(x,2A_0C_0\delta^l)}\sum_{k=l}^\infty\delta^{-ksq}
\sum_{i=0}^\infty\delta^{i\gamma q} \delta^{(k-i)\varepsilon'q}\\
&\quad\quad\times\sum_{j=k-i-2}^\infty\delta^{j(s-\varepsilon')q}
\left\{\oint_{B(z,\delta^{k-i-1})}[g_j(y)]^{\frac{\omega}{\omega+\varepsilon}}
\,d\mu(y)\right\}^{\frac{(\omega+\varepsilon)q}{\omega}}\,d\mu(z)\\
&\quad \sim\oint_{B(x,2A_0C_0\delta^l)}\sum_{k=l}^\infty\delta^{-ksq}
\sum_{m=-\infty}^k\delta^{(k-m)\gamma q} \delta^{m\varepsilon'q}\\
&\quad\quad\times\sum_{j=m-2}^\infty\delta^{j(s-\varepsilon')q}
\left\{\oint_{B(z,\delta^{m-1})}[g_j(y)]^{\frac{\omega}{\omega+\varepsilon}}
\,d\mu(y)\right\}^{\frac{(\omega+\varepsilon)q}{\omega}}\,d\mu(z)\\
&\quad\lesssim \oint_{B(x,2A_0C_0\delta^l)}\sum_{m=-\infty}^\infty
\delta^{(l\vee m)(\gamma-s)q}\delta^{m(\varepsilon'-\gamma)q}\\
&\quad\quad\times\sum_{j=m-2}^\infty\delta^{j(s-\varepsilon')q}
\left\{\oint_{B(z,\delta^{m-1})}[g_j(y)]^{\frac{\omega}{\omega+\varepsilon}}
\,d\mu(y)\right\}^{\frac{(\omega+\varepsilon)q}{\omega}}\,d\mu(z)\\
&\quad \lesssim \oint_{B(x,2A_0C_0\delta^l)}\sum_{m=-\infty}^l
\delta^{l(\gamma-s)q}\delta^{m(\varepsilon'-\gamma)q}\\
&\quad\quad\times\sum_{j=m-2}^\infty\delta^{j(s-\varepsilon')q}
\left\{\oint_{B(z,\delta^{m-1})}[g_j(y)]^{\frac{\omega}{\omega+\varepsilon}}
\,d\mu(y)\right\}^{\frac{(\omega+\varepsilon)q}{\omega}}\,d\mu(z)\\
&\quad\quad+ \oint_{B(x,2A_0C_0\delta^l)}
\sum_{m=l+1}^\infty\delta^{m(\varepsilon'-s)q}\cdots\\
&\quad=:\mathrm{Y}_3+\mathrm{Y}_4.
\end{align*}
To estimate $\mathrm{Y}_3$, choosing $\varepsilon\in(0,s)$
such that $\omega/(\omega+s)<\omega/(\omega+\varepsilon)<q$,
and using the H\"older inequality and \eqref{ch-g}, we find that,
for any $m\in\zz$ and $z\in\cx$,
$$\left\{\oint_{B(z,\delta^{m-1})}[g_j(y)]^{\frac{\omega}{\omega+\varepsilon}}
\,d\mu(y)\right\}^{\frac{(\omega+\varepsilon)q}{\omega}}
\leq \oint_{B(z,\delta^{m-1})}[g_j(y)]^q\,d\mu(y)
\lesssim \|u\|_{\dot{M}^s_{\infty,q}(\cx)}^q,$$
which implies that
$$\mathrm{Y}_3\lesssim \|u\|_{\dot{M}^s_{\infty,q}(\cx)}^q
\sum_{m=-\infty}^l\delta^{l(\gamma-s)q}\delta^{m(\varepsilon'-\gamma)q}
\sum_{j=m-2}^\infty\delta^{j(s-\varepsilon')q}
\lesssim \|u\|_{\dot{M}^s_{\infty,q}(\cx)}^q.$$
To estimate $\mathrm{Y}_4$, observe that, for any $m\in\zz$ with $m\geq l+1$,
$z\in B(x,2A_0C_0\delta^l)$, and $y\in B(z,\delta^{m-1})$, we have
$$d(y,x)\leq A_0[d(y,z)+d(z,x)]<A_0\delta^{m-2}+2A_0^2C_0\delta^l
\leq(A_0\delta+2A_0^2C_0\delta^2)\delta^{l-2}\leq\delta^{l-2},$$
which further implies that
\begin{align*}
\mathrm{Y}_4&\lesssim \oint_{B(x,2A_0C_0\delta^l)}
\sum_{m=l+1}^\infty\delta^{m(\varepsilon'-s)q}
\sum_{j=m-2}^\infty\delta^{j(s-\varepsilon')q}\\
&\quad\times\left\{\oint_{B(z,\delta^{m-1})}[g_j(y)
\mathbf{1}_{B(x,\delta^{l-2})}(y)]^{\frac{\omega}{\omega+\varepsilon}}
\,d\mu(y)\right\}^{\frac{(\omega+\varepsilon)q}{\omega}}\,d\mu(z)\\
&\lesssim \oint_{B(x,2A_0C_0\delta^l)}\sum_{m=l+1}^\infty
\delta^{m(\varepsilon'-s)q}\sum_{j=m-2}^\infty\delta^{j(s-\varepsilon')q}
\left\{M\left([g_j\mathbf{1}_{B(x,\delta^{l-2})}]^{\frac{\omega}{\omega+\varepsilon}}\right)(z)
\right\}^{\frac{(\omega+\varepsilon)q}{\omega}}\,d\mu(z)\\
&\lesssim \oint_{B(x,2A_0C_0\delta^l)}\sum_{j=l-1}^\infty
\left\{M\left([g_j\mathbf{1}_{B(x,\delta^{l-2})}]^{\frac{\omega}{\omega+\varepsilon}}\right)(z)
\right\}^{\frac{(\omega+\varepsilon)q}{\omega}}\,d\mu(z)\\
&\lesssim \sum_{j=l-2}^\infty\oint_{B(x,\delta^{l-2})}
\left[g_j(z)\right]^q\,d\mu(z)\lesssim \|u\|_{\dot{M}^s_{\infty,q}(\cx)}^q.
\end{align*}
Combining the estimates of $\mathrm{Y}_3$ and $\mathrm{Y}_4$,
we conclude that, for any $l\in\zz$ and $x\in\cx$,
$$
\oint_{B(x,2A_0C_0\delta^l)}\sum_{k=l}^\infty\delta^{-ksq}
\sup_{\phi\in\cf_k(z)}|\langle u,\phi\rangle|^q\,d\mu(z)
\lesssim \|u\|_{\dot{M}^s_{\infty,q}(\cx)}^q,
$$
which, together with Lemma~\ref{2-cube}(iii), implies the desired estimate in this case.

Thus, we have $\dot{M}^s_{p,q}(\cx)\subset\ca\hf$.

We finally show $\ca\hf\subset\dot{M}^s_{p,q}(\cx)$. Assume that $f\in\ca\hf$.
By Lemma \ref{btl-loc} and its proof, we know that there
exist a subsequence $\{Q_{k_j}f\}_{j\in\nn}$ and a constant $C\in\rr$
such that, for almost every $x\in\cx$, $\lim_{j\to\infty}Q_{k_j}f(x)=f(x)-C$.
For any $k\in\zz$ and $x\in\cx$, let
$$g_k(x):=\delta^{-ks}\sup_{\phi\in\cf_k(x)}|\langle f,\phi\rangle|.$$
For almost every $x,\ y\in\cx$ and $k_0\in\zz$ satisfying
$\delta^{k_0+1}\leq d(x,y)<\delta^{k_0}$, we can find a $j_0\in\nn$ satisfying
$k_0\leq k_{j_0}$ and we can estimate
\begin{align*}
|f(x)-f(y)|
&=\left|f(x)-Q_{k_{j_0}}f(x)+Q_{k_{j_0}}f(x)
-Q_{k_{j_0}}f(y)+Q_{k_{j_0}}f(y)-f(y)\right|\\
&\leq \left|Q_{k_{j_0}}f(x)-Q_{k_{j_0}}f(y)\right|
+\sum_{j=j_0}^\infty \left[\left|Q_{k_{j+1}}f(x)-Q_{k_j}f(x)\right|
+\left|Q_{k_{j+1}}f(x)-Q_{k_j}f(x)\right|\right]\\
&\leq \delta^{k_{j_0}s}[g_{k_{j_0}}(x)+g_{k_{j_0}}(y)]
+\sum_{k=k_0}^\infty \left[\left|Q_{k+1}f(x)-Q_kf(x)\right|
+\left|Q_{k+1}f(x)-Q_{k}f(x)\right|\right]\\
&\leq 2\sum_{k=k_0}^\infty\delta^{ks}[g_k(x)+g_k(y)].
\end{align*}
For any $k\in\zz$, define
$$h_k:=2\sum_{j=k}^\infty\delta^{(-k+j-1)s}g_j.$$
We then have, for almost every $x,\ y\in\cx$
with  $\delta^{k+1}\leq d(x,y)<\delta^k$,
\begin{align*}
|f(x)-f(y)|&\leq 2\sum_{j=k}^\infty\delta^{js}[g_j(x)+g_j(y)]
\leq \delta^{(k+1)s}[h_k(x)+h_k(y)]\\
&\leq [d(x,y)]^s[h_k(x)+h_k(y)],
\end{align*}
which implies that $\{h_k\}_{k\in\zz}\in\dd^s(f)$.
Note that, by the H\"older inequality when $q\in(1,\infty]$,
or \eqref{r} when $q\in(\omega/(\omega+s),1]$, we obtain
\begin{align*}
\sum_{k=-\infty}^\infty h_k^q
\lesssim \sum_{k=-\infty}^\infty
\left[\sum_{j=k}^\infty\delta^{(-k+j-1)s}g_j\right]^q
\lesssim\sum_{k=-\infty}^\infty\sum_{j=k}^\infty\delta^{(-k+j-1)sq/2}g_j^q
\lesssim \sum_{j=-\infty}^\infty g_j^q,
\end{align*}
which implies that, for any given $p \in(\omega/(\omega+s),\infty)$
and $q \in(\omega/(\omega+s),\infty]$,
\begin{align*}
\|f\|_{\dot{M}^s_{p,q}(\cx)}
\leq \left\|\left(\sum_{k=\infty}^\infty h_k^q\right)^{1/q}\right\|_{L^p(\cx)}
\lesssim \left\|\left(\sum_{j=\infty}^\infty g_j^q\right)^{1/q}\right\|_{L^p(\cx)}
\sim \|f\|_{\ca\hf}.
\end{align*}
When $p=\infty$, then, by the H\"older inequality when $q\in(1,\infty]$, or
\eqref{r} when $q\in(\omega/(\omega+s),1)$, we have, for any $l\in\zz$,
\begin{align*}
\sum_{k=l}^\infty h_k^q
\lesssim \sum_{k=l}^\infty\left[\sum_{j=k}^\infty\delta^{(-k+j-1)s}g_j\right]^q
\lesssim\sum_{k=l}^\infty\sum_{j=k}^\infty\delta^{(-k+j-1)sq/2}g_j^q
\lesssim \sum_{j=l}^\infty g_j^q,
\end{align*}
which implies that, for any $x\in\cx$,
$$\sum_{k=l}^\infty\oint_{B(x,\delta^l)} [h_k(y)]^q\,d\mu(y)
\lesssim  \sum_{j=l}^\infty \oint_{B(x,\delta^l)}[g_j(y)]^q\,d\mu(y).$$
From this, we deduce that $\|f\|_{\dot{M}^s_{\infty,q}(\cx)}\lesssim \|f\|_{\ca\hfi}$,
which completes the proof of (i)
and hence of Theorem \ref{hbtl-gbtl}.
\end{proof}

Now, we can show Theorem \ref{mr}.
\begin{proof}[Proof of Theorem \ref{mr}]
The Theorem \ref{mr} is a direct corollary of Theorems \ref{btl-gbtl} and
\ref{hbtl-gbtl}; we omit the details.
\end{proof}

\section{Pointwise characterization of inhomogeneous
Besov and Triebel--Lizorkin spaces}\label{s4}

In this section, we establish the inhomogeneous
version of Theorem  \ref{mr}.
Let us begin with the notion of inhomogeneous approximations
of the identity with exponential decay (see \cite[Definition 6.1]{hlyy}).

\begin{definition}\label{iati}
Let $\eta\in(0,1)$ be as in Definition \ref{exp-ati}.
A sequence $\{Q_k\}_{k\in\zz_+}$ of bounded linear integral operators on
$L^2(\cx)$ is called an  \emph{inhomogeneous approximation
of the identity with exponential decay} (for short,
exp-IATI) if $\{Q_k\}_{k\in\zz_+}$ has the following properties:
\begin{enumerate}
\item[{\rm(i)}] $\sum_{k=0}^\infty Q_k=I$ in $L^2(\cx)$;
\item[{\rm(ii)}] for any $k\in\nn$, $Q_k$
satisfies (ii) through (v) of Definition \ref{exp-ati};
\item[{\rm(iii)}] $Q_0$ satisfies (ii), (iii), and (iv) of Definition
\ref{exp-ati} with $k=0$ but without the decay factor
$$\exp\left\{-\nu\left[\max\left\{d(x,\cy^0),d(y,\cy^0)\right\}\right]^a\right\};$$
moreover, for any $x\in\cx$,
\begin{equation}\label{e6.1}
\int_\cx Q_0(x,y)\,d\mu(y)=1=\int_\cx Q_0(y,x)\,d\mu(y).
\end{equation}
\end{enumerate}
\end{definition}

\begin{remark} As was pointed out in \cite[Remark 6.2]{hlyy},
the existence of an exp-IATI on $\cx$ is guaranteed by the
main results from \cite{ah13}.
In Definition \ref{iati}, due to \eqref{e6.1},
we do not need $\diam \cx = \infty$ to guarantee the existence of an exp-IATI on $\cx$.
In other words, $\diam \cx$ can be finite or infinite.
\end{remark}

Based on the notion of exp-IATIs,
He et al. established the
following inhomogeneous discrete Calder\'on
reproducing formulae in \cite[Theorems 6.10 and 6.13]{hlyy}.

\begin{lemma}\label{icrf}
Let $\{Q_k\}_{k\in\zz_+}$ be an {\rm exp-IATI} and $\beta,\ \gamma\in (0, \eta)$ with
$\eta$ as in Definition \ref{exp-ati}.
For any $k\in\zz_+,$ $\alpha\in\ca_k$, and
$m\in\{1,\dots,N(k,\alpha)\}$,
suppose that $\ya$ is an arbitrary point in $\qa$.
Then there exist an $N\in\nn$ and a
sequence $\{\widetilde{Q}_k\}_{k\in\zz_+}$ of bounded
linear integral operators on $L^2(\mathcal{X})$ such
that, for any $f \in (\icgg)'$,
\begin{align*}
f(\cdot) &= \sum_{\alpha \in \ca_0}\sum_{m=1}^{N(0,\alpha)}\int_{\qo}
\widetilde{Q}_0(\cdot,y)\,d\mu(y)Q^{0,m}_{\alpha,1}(f)\\
&\quad+\sum_{k=1}^N\sum_{\alpha \in \ca_k}\sum_{m=1}^{N(k,\alpha)}
\mu\left(\qa\right)\widetilde{Q}_k(\cdot,\ya)Q^{k,m}_{\alpha,1}(f)\\
&\quad+\sum_{k=N+1}^\infty\sum_{\alpha \in \ca_k}\sum_{m=1}^{N(k,\alpha)}
\mu\left(\qa\right)\widetilde{Q}_k(\cdot,\ya)Q_kf\left(\ya\right)
\end{align*}
in $(\icgg)',$ where, for any $k\in\{0,\ldots,N\}$, $\alpha\in\ca_k$,
$m\in\{1,\dots,N(k,\alpha)\}$,
and $x\in\cx$,
$$Q^{k,m}_{\alpha,1}(x):=\frac{1}{\mu(\qa)}\int_{\qa}Q_k(y,x)\,d\mu(y),$$
and $Q^{k,m}_{\alpha,1}(f):=\langle f, Q^{k,m}_{\alpha,1}\rangle$.
Moreover, for any $k\in\zz_+$, the kernel of $\widetilde{Q}_k$,
still denoted by $\widetilde{Q}_k$,  satisfies
\eqref{4.23x}, \eqref{4.23y}, and the
following integral condition: for any $x\in\mathcal{X}$,
$$\int_{\mathcal{X}}\widetilde{Q}_k(x,y)\,d\mu(y)=
\int_{\mathcal{X}}\widetilde{Q}_k(y,x)\,d\mu(y)=\begin{cases}
1 &\text{if } k \in \{0,\dots,N\},\\
0 &\text{if } k\in \{N+1,N+2,\ldots\}.
\end{cases}$$
\end{lemma}

Now, we recall the notions of inhomogeneous
spaces $\ihb$ and $\ihf$ introduced in \cite{whhy}.
To this end, for any dyadic cube $Q$ and any
non-negative measurable function $f$ on $X$, let
\begin{equation*}
m_Q(f):=\frac 1{\mu(Q)}\int_Q f(y)\,d\mu(y).
\end{equation*}

\begin{definition}\label{ih}
Let  $\beta, \gamma \in (0, \eta)$ with $\eta$ as in Definition \ref{exp-ati},
and $s\in(-(\beta\wedge\gamma),\beta\wedge\gamma)$.
Let $\{Q_k\}_{k\in\zz_+}$ be an exp-IATI and $N\in\nn$ as in Lemma \ref{icrf}.
\begin{enumerate}
\item[{\rm(i)}] If
$p\in(p(s,\beta\wedge\gamma),\infty]$ with $p(s,\beta\wedge\gamma)$
as in \eqref{pseta}, and $q\in(0,\fz]$, then
the \emph{inhomogeneous Besov space $B^s_{p,q}(\cx)$}
is defined to be the set of all $f\in(\icgg)'$ such that
\begin{align*}
\|f\|_{B^{s}_{p,q}(\cx)}:=&\,\lf\{\sum_{k=0}^N\sum_{\az\in\ca_k}
\sum_{m=1}^{N(k,\az)}\mu\lf(Q_\az^{k,m}\r)
\lf[m_{Q_\az^{k,m}}(|Q_kf|)\r]^p\r\}^{1/p}\\
&\quad+\lf[\sum_{k=N+1}^\fz\dz^{-ksq}\|Q_kf\|_{L^p(\cx)}^q\r]^{1/q}\\
<&\,\fz
\end{align*}
with the usual modifications made when $p=\fz$ or $q=\fz$.
\item[{\rm(ii)}] If
$p\in(p(s,\bz\wedge\gz),\fz)$ and
$q\in(p(s,\bz\wedge\gz),\fz]$, then the \emph{inhomogeneous
Triebel--Lizorkin space $F^s_{p,q}(\cx)$} is defined
to be the set of all $f\in(\icgg)'$ such that
\begin{align*}
\|f\|_{F^{s}_{p,q}(\cx)}:=&\lf\{\sum_{k=0}^N\sum_{\az\in\ca_k}
\sum_{m=1}^{N(k,\az)}\mu\lf(Q_\az^{k,m}\r)
\lf[m_{Q_\az^{k,m}}(|Q_kf|)\r]^p\r\}^{1/p}\\
&\quad+\lf\|\lf(\sum_{k=N+1}^\fz\dz^{-ksq}|Q_kf|^q\r)^{1/q}\r\|_{L^p(\cx)}\\
<&\,\fz
\end{align*}
with the usual modification made when $q=\fz$.
\end{enumerate}
\end{definition}

\begin{remark}
\begin{enumerate}
\item[{\rm(i)}] We point out that we \emph{do not} need the
assumption $\mu(\cx)=\infty$ in Definitions \ref{iati} and \ref{ih}.
\item[{\rm(ii)}] It was proved in \cite[Propositions 4.3 and 4.4]{whhy} that,
when $\beta$, $\gamma$, $s$, $p$, and $q$ are as in Definition~\ref{ih},  the
inhomogeneous Besov and Triebel--Lizorkin spaces are
independent of the choices of both  exp-IATIs and
the spaces of distributions.
\end{enumerate}
\end{remark}

We next recall the notions of 1-exp-IATIs (see, for instance, \cite[Definition 3.1]{hyy19})
and the local Hardy space $h^p(\cx)$ (see, for instance, \cite[Section 3]{hyy19}).

\begin{definition}\label{1-exp-iati}
Let $\eta\in(0,1)$ be as in Definition \ref{exp-ati}.
A sequence $\{P_k\}_{k\in\zz_+}$ of bounded linear integral operators
on $L^2(\cx)$ is called  an \emph{inhomogeneous
approximation of the identity with exponential
decay and integration 1} (for short, 1-exp-IATI) if $\{P_k\}_{k\in\zz_+}$
has the following properties:
\begin{enumerate}
\item[{\rm(i)}] for any $k\in\zz_+$, $P_k$
satisfies (ii) and (iii) of Definition \ref{exp-ati} but without the term
$$\exp\left\{-\nu\left[\max\{d(x,\cy^k),d(y,\cy^k)\}\right]^a\right\};$$
\item[{\rm(ii)}] for any $k\in\zz_+$ and $x\in\cx$,
$$\int_\cx P_k(x,y)\,d\mu(y)=1=\int_\cx P_k(y,x)\,d\mu(y);$$
\item[{\rm(iii)}] letting $Q_0:=P_0$ and, for any $k\in\nn$,
$Q_k:=P_k-P_{k-1}$, then $\{Q_k\}_{k\in\zz_+}$ is an exp-IATI.
\end{enumerate}
\end{definition}

\begin{definition}\label{h1}
Let $\cx$ be a space of homogeneous type.
Let $\{P_k\}_{k\in\zz}$ be a 1-exp-IATI.
The \emph{local radial maximal function} $\cm^+_0(f)$ of
$f$ is defined by setting, for any $x\in\cx$,
$$\cm^+_0(f)(x):=\max\left\{\max_{k\in\{0,\dots,N\}}
\left\{\sum_{\alpha\in\ca_k}
\sum_{m=1}^{N(k,\alpha)}\sup_{z\in\qa}|P_kf(z)|\mathbf{1}_{\qa}(x)\right\},
\sup_{k\in\{N+1,N+2,\dots\}}|P_kf(x)|\right\},$$
where $N\in\nn$ is as in Lemma \ref{icrf}.
For any $p\in(0,\infty)$, the \emph{local Hardy space} $h^p(\cx)$ is defined by setting
$$h^p(\cx):=\left\{f\in(\icgg)':\ \|f\|_{h^p(\cx)}:=
\|\cm^+_0(f)\|_{L^p(\cx)}<\infty\right\}.$$
\end{definition}

\begin{remark}
In \cite[Theorem 3.3]{hyy19}, it was shown that, when $p\in(1,\infty]$,
then $h^p(\cx)=L^p(\cx)$. Besides, in \cite[Theorem 6.13]{whhy},
it was  proved that, when $p\in(1,\infty)$, $F^0_{p,2}=L^p(\cx)$.
Moreover,  the Littlewood--Paley $g$-function characterization of $h^p(\cx)$
in \cite[Theorem 5.7]{hyy19} implies that,
for any given $p\in (\omega/(\omega+\eta),1]$,
$F_{p,2}^0(\cx)=h^p(\cx)$.
\end{remark}

Now, we introduce the  notions of inhomogeneous
Haj\l asz--Sobolev spaces, Haj\l asz--Triebel--Lizorkin spaces,
and Haj\l asz--Besov spaces.

\begin{definition}\label{ih-s-btl}
Let $s\in(0,\infty)$.
\begin{enumerate}
\item[{\rm(i)}] Let $p\in(0,\infty)$.
The \emph{inhomogeneous Haj\l asz--Sobolev space}
$M^{s,p}(\cx)$ is defined to be the set of all
measurable functions $u$ on $\cx$ such that
$$\|u\|_{M^{s,p}(\cx)}:=\|u\|_{h^p(\cx)}+\|u\|_{\dot{M}^{s,p}(\cx)}<\infty.$$
\item[{\rm(ii)}] Let $p,\  q\in(0,\infty]$. The \emph{inhomogeneous
Haj\l asz--Triebel--Lizorkin space}
$M^{s}_{p,q}(\cx)$ is defined to be the set of all measurable
functions $u$ on $\cx$ such that
$$\|u\|_{\dot{M}^s_{p,q}(\cx)}:=\|u\|_{h^p(\cx)}+\|u\|_{\dot{M}^s_{p,q}(\cx)}<\infty.$$
\item[{\rm(iii)}] Let $p,\ q\in(0,\infty]$. The \emph{inhomogeneous Haj\l asz--Besov space}
$N^s_{p,q}(\cx)$ is defined to be the set of all measurable
functions $u$ on $\cx$ such that
$$\|u\|_{N^s_{p,q}(\cx)}:=\|u\|_{h^p(\cx)}+\|u\|_{\dot{N}^s_{p,q}(\cx)}.$$
\end{enumerate}
\end{definition}

The following theorem states the inhomogeneous version of Theorem \ref{mr}.

\begin{theorem}\label{imr}
Let $\omega$ and $\eta$ be, respectively,
as in \eqref{eq-doub} and Definition \ref{exp-ati},
$\beta,\ \gamma\in(0,\eta)$, $s\in(0,\beta\wedge\gamma)$,
and $p,\ q$ be as in  Definition~\ref{ih}.
Assume that $\mu(\cx)=\infty$ and the measure $\mu$ of $\cx$
has a weak lower bound $Q=\omega$.
\begin{enumerate}
\item[{\rm(i)}] If $p\in (\omega/(\omega+s),\infty)$
and $q\in (\omega/(\omega+s),\infty]$, then $M^s_{p,q}(\cx)=\ihf$.
\item[{\rm(ii)}] If $p\in (\omega/(\omega+s),\infty]$
and $q\in(0,\infty]$, then $N^s_{p,q}(\cx)=\ihb$.
\end{enumerate}
\end{theorem}

To prove Theorem \ref{imr}, we first establish a
relationship between local Hardy spaces and inhomogeneous
Besov and Triebel--Lizorkin spaces, whose RD-space
version was obtained in \cite[Theorem 1.2]{yz11}. We point out that the proof of
\cite[Theorem 1.2]{yz11} just depends on the size, the regularity,
and the cancellation conditions of the approximation of the identity (for short, ATI),
but \emph{does not} involve the reverse doubling condition of
the underlying space and the bounded support of ATI. This results in
that the proof of \cite[Theorem 1.2]{yz11} is still valid in any space of homogeneous type due
to Definition \ref{exp-ati}(v) and Lemma \ref{pro-qk}; we omit the details.

\begin{theorem}\label{h-btl}
Let $\beta,\ \gamma\in (0,\eta)$ with $\eta$ as in
Definition \ref{exp-ati}, $s\in(0,\beta\wedge\gamma)$, and
$p\in(\omega/(\omega+s),\infty)$ with
$\omega$ as in \eqref{eq-doub}.
Assume $\mu(\cx)=\infty$. Let $\{Q_k\}_{k\in\zz}$ be an exp-ATI.
\begin{enumerate}
\item[{\rm(i)}] If $q\in(0,\infty]$,
then $f\in\ihb$ if and only if $f\in h^p(\cx)$ and
$${\rm J}_1:=\lf[\sum_{k=-\infty}^\fz\dz^{-ksq}\|Q_kf\|_{L^p(X)}^q\r]^{1/q}<\fz.$$
Moreover, $\|f\|_{\ihb}$ is equivalent to $\|f\|_{h^p(\cx)}+{\rm J}_1$ with positive
equivalence constants independent of $f$.
\item[{\rm(ii)}] If $q\in(\omega/(\omega+s),\infty]$, then $f\in\ihf$ if and only if $f\in h^p(\cx)$ and
$${\rm J}_2:=\lf\|\lf(\sum_{k=-\infty}^\fz\dz^{-ksq}|Q_kf|^q\r)^{1/q}\r\|_{L^p(X)}<\fz.$$
Moreover, $\|f\|_{\ihf}$ is equivalent to $\|f\|_{h^p(\cx)}+{\rm J}_2$
with positive equivalence constants independent of $f$.
\end{enumerate}
\end{theorem}

\begin{remark}
Usually, it makes no sense to write the conclusions
of Theorem \ref{h-btl} as $\ihb=h^p(\cx)\cap\hb$
and $\ihf=h^p(\cx)\cap\hf$ because  homogeneous
and inhomogeneous spaces are defined via different
kinds of spaces of distributions
(see \cite[Remark 1.1(iv)]{yz11}).
\end{remark}

Now, we show Theorem \ref{imr}.

\begin{proof}[Proof of Theorem \ref{imr}]
We only prove (i) because the proof of (ii) is similar to that of (i).
We first show $M^s_{p,q}(\cx)\subset\ihf$.
To this end, assume that $u\in M^s_{p,q}(\cx)$.
By Definition~\ref{ih-s-btl}, we know that $u\in h^p(\cx)$ and
$u\in \dot{M}^s_{p,q}(\cx)$.
We consider two cases on $p$.

{\it Case 1)} $p\in (1, \infty)$.
In this case, since $p\in (1,\infty)$, from  \cite[Theorem 3.3]{hyy19},
it follows that $u\in L^p(\cx)$.
Moreover, by Theorem \ref{mr}, we know that
$u\in \hf$. These, together with \cite[Theorem 6.12]{whhy}, imply that $u\in\ihf$ and
$\|u\|_{\ihf}\lesssim \|u\|_{M^s_{p,q}(\cx)}$.

{\it Case 2)} $p\in (\omega/(\omega+s), 1]$. In this case,
by Theorem \ref{mr}, we know that $u\in \hf$ and
$\|u\|_{\hf}\lesssim \|u\|_{\dot{M}^s_{p,q}(\cx)}$.
Then, using Theorem \ref{h-btl}(ii), we conclude that
$u\in\ihf$ and $\|u\|_{\ihf}\lesssim \|u\|_{M^s_{p,q}(\cx)}$.

Next, we show $\ihf\subset M^s_{p,q}(\cx)$.
To this end, assume that $u\in \ihf$. We also consider two cases on $p$.

{\it Case 1)} $p\in (1, \infty)$.
In this case, since $p\in (1,\infty)$,
from \cite[Theorem 6.12]{whhy}, it follows that $u\in L^p(\cx)\cap \hf$ and
$$\|u\|_{\ihf}\sim \|u\|_{L^p(\cx)}+\|u\|_{\hf}.$$
By \cite[Theorem 3.3]{hyy19} again, we know
that $u\in h^p(\cx)$ and $\|u\|_{h^p(\cx)}\sim\|u\|_{L^p(\cx)}$.
Besides, from Theorem \ref{mr}, we deduce  that
$u\in \dot{M}^s_{p,q}(\cx)$ and $\|u\|_{\dot{M}^s_{p,q}(\cx)}\sim\|u\|_{\hf}$,
which further  implies that $u\in M^s_{p,q}(\cx)$ and
$\|u\|_{M^s_{p,q}(\cx)}\lesssim \|u\|_{F^s_{p,q}(\cx)}$.

{\it Case 2)}  $p\in (\omega/(\omega+s), 1]$.
In this case, since $u\in\ihf$, from  \cite[Proposition 4.4]{whhy},
it follows that $u\in(\icgg)'$ with $\beta$ and $\gamma$ as in Definition~\ref{ih}.
By $p\in (\omega/(\omega+s), 1]$,
we know that $s>\omega(1/p-1)$.
Choosing $\beta_0\in(0,\eta)$ and
$\gamma_0\in (s,\eta)\subset(\omega[1/p-1],\eta)$,
then we find that $u\in(\cg_0^\eta(\beta_0,\gamma_0))'
\subset (\mathring{\cg}_0^\eta(\beta_0,\gamma_0))'$.
From this, Theorem \ref{h-btl}(ii), and \cite[Proposition 3.15]{whhy},
we deduce that $u\in\hf$ and
$$\|u\|_{\ihf}\sim \|u\|_{L^p(\cx)}+\|u\|_{\hf},$$
which implies that $u\in M^s_{p,q}(\cx)$ and
$\|u\|_{M^s_{p,q}(\cx)}\lesssim \|u\|_{F^s_{p,q}(\cx)}$.
This finishes the proof of (i) and hence of Theorem \ref{imr}.
\end{proof}

\begin{remark}
We point out that it is not clear whether or not Theorem \ref{imr}
still holds true when $\mu(\cx)\neq\infty$, due to that the
existence of exp-ATIs in Theorem \ref{h-btl} needs $\mu(\cx)=\infty$
[see Remark~\ref{2.8x}(i)].
\end{remark}

\paragraph{Acknowledgments.} The second author would like to
thank Ziyi He for his valuable discussions and suggestions
for  this article.

\bigskip

\noindent Ryan Alvarado

\medskip

\noindent Department of Mathematics and Statistics, Amherst College, 303 Seeley Mudd,
Amherst, MA 01002, United States of America

\smallskip

\noindent{\it E-mail:} \texttt{rjalvarado@amherst.edu}

\bigskip

\noindent Fan Wang, Dachun Yang (Corresponding author) and Wen Yuan

\medskip

\noindent Laboratory of Mathematics and Complex Systems (Ministry of Education of China),
School of Mathematical Sciences, Beijing Normal University, Beijing 100875, People's Republic of China

\smallskip

\noindent{\it E-mails:} \texttt{fanwang@mail.bnu.edu.cn} (F. Wang)

\noindent\phantom{{\it E-mails:} }\texttt{dcyang@bnu.edu.cn} (D. Yang)

\noindent\phantom{{\it E-mails:} }\texttt{wenyuan@bnu.edu.cn} (W. Yuan)

\bigskip

\end{document}